\documentclass[12pt]{amsart}

\usepackage[american]{babel}
\usepackage{dsfont,mathtools,amssymb}
\usepackage{color}

\mathtoolsset{showonlyrefs,showmanualtags}

\theoremstyle{definition}
\newtheorem{theorem}{Theorem}[section]
\newtheorem{proposition}[theorem]{Proposition}
\newtheorem{lemma}[theorem]{Lemma}

\newtheorem{remark}[theorem]{Remark}
\newtheorem{definition}[theorem]{Definition}

\newtheorem{notation}[theorem]{Notation}
\newtheorem*{acknowledgement}{Acknowledgement}

\numberwithin{equation}{section}
\numberwithin{figure}{section}

\newcommand{\dx}{\mathrm{d}}
\newcommand{\e}{\mathrm{e}}
\newcommand{\E}{\mathds{E}}

\newcommand{\setone}{\mathds{1}}
\renewcommand{\rho}{\varrho}
\renewcommand{\phi}{\varphi}

\DeclareMathOperator{\var}{var}
\DeclareMathOperator{\cov}{cov}

\DeclareMathOperator{\dist}{dist}

\title[Decay of correlations in 1D lattice systems]{Decay of correlations in 1D lattice systems of continuous spins and long-range interaction}

\date{September 1, 2013}

\author{Georg Menz }
\address{Georg Menz\\Stanford University}
\email{gmenz@stanford.edu}
\author{Robin Nittka}
\address{Robin Nittka\\Max Planck Institute for Mathematics in the Sciences\\Inselstr. 22\\04103 Leipzig\\Germany}
\email{robin.nittka@gmail}

\subjclass[2000]{Primary 82C26; secondary 82B20; 60K35; 26D10.}

\keywords {lattice systems, continuous spin, long-range interaction, decay of correlations, Gibbs measure, phase transition, logarithmic Sobolev inequality}

\begin{document}

\begin{abstract}
  We consider an one-dimensional lattice system of unbounded and continuous spins. The Hamiltonian consists of a perturbed strictly-convex single-site potential and with longe-range interaction. We show that if the interactions decay algebraically of order $2+\alpha$, $\alpha>0$ then the correlations also decay algebraically of order $2+ \tilde \alpha$ for some $\tilde \alpha > 0$. For the argument we generalize a method from Zegarlinski from finite-range to infinite-range interaction to get a preliminary decay of correlations, which is improved to the correct order by a recursive scheme based on Lebowitz inequalities. Because the decay of correlations yields the uniqueness of the Gibbs measure, the main result of this article yields that the on-phase region of a continuous spin system is at least as large as for the Ising model.
\end{abstract}

\maketitle

\section{Introduction and main results} 

We consider a lattice system of unbounded and continuous spins on the one-dimensional lattice $\mathds{Z}$. The formal Hamiltonian $H: \mathds{R}^{\mathds{Z}} \to \mathds{R}$ of the system is given by  
\begin{equation}
  \label{e_d_Hamiltonian}
	H(x) = \sum_{i \in \mathds{Z} } \psi_i(x_i) + \frac{1}{2} \sum_{i,j \in \mathds{Z}} M_{ij} x_i x_j.  
\end{equation}
We assume that the single-site potentials $\psi_i : \mathds{R} \to \mathds{R}$ are perturbed convex. This means that there is a splitting $\psi_i= \psi_i^c + \psi_i^b$ such that for all $i \in \mathds{Z}$ and $z \in \mathds{R}$
\begin{equation}\label{e_cond_psi}
  (\psi_i^c )'' (z) \geq 0 \qquad \mbox{and} \qquad |\psi_i^b (z)| + \left| (\psi_i^b)' (z) \right| \lesssim 1.
\end{equation}
Here, we used the convention that (see also Definition~\ref{def:dep} from below)
\begin{equation*}
  a \lesssim b \qquad :\Leftrightarrow \mbox{there is a uniform constant $C>0$ such that $a \leq C b$}.
 \end{equation*}
Moreover, we assume that 
\begin{itemize}
\item  the interaction is symmetric i.e.~
  \begin{equation}
    \label{e_ass_sym}
    M_{ij}=M_{ji} \qquad \mbox{ for all $i, j \in \mathds{Z}$,}
  \end{equation}
  
\item and the matrix $M= (M_{ij})$ is strictly diagonal dominant i.e.~for some $\delta > 0$ it holds for any $i \in \mathds{Z}$
\begin{equation}\label{e_strictly_diag_dominant}
  \sum_{j \in \mathds{Z}, j \neq i} |M_{ij}| + \delta \le M_{ii}. 
\end{equation}
\end{itemize}

One of the most interesting phenomenon in statistical mechanics are phase
transitions. In the case of the one-dimensional Ising model
(i.e.~the spin values $x_i$ are either $-1$ or $1$) there is
a unique infinite-volume Gibbs measure, if the interaction
decays sufficiently fast i.e. 
 \begin{equation} \label{e_decay_ising}
|M_{ij}| \lesssim \frac{1}{|i-j|^{2 + \alpha}+1}  
\end{equation}
uniformly in $i,j \in \mathds{Z}$ for some $\alpha>0$ (cf.~\cite{Dobru_1,Dobru_2,Ruelle_1,CaMaOrPi}). Hence, there is no phase transition for such decay of interaction. However, if the interaction decays sufficiently slowly i.e.~$-1 <\alpha \leq 0$, it is known that the one-dimensional Ising model has a phase transition (cf.~\cite{Dyson,FroeSpenc,CaFeMePr,Imbrie}). \medskip

Though one would expect similar results to hold in the case of continuous and unbounded spins, surprisingly few facts are known. This is surprising because one could think that technically the task becomes easier: Due to the continuity of the spins one can use analytic tools like differentiation and convexity. However, in our situation there also is a technical drawback: Due to the unboundedness of spin values one looses compactness. Hence, many arguments known from the bounded discrete case do not carry over.   
\medskip

In fact, the only rigorous results obtained to date are for finite-range interaction, i.e.~there is an integer $R$ such that 
\begin{equation*}
M_{ij}=0 \qquad  \qquad \mbox{if } \quad |i-j|\geq R .
\end{equation*}
 Zegarlinski showed that the spin-spin correlation function decays exponentially fast (cf.~\cite[Lemma~4.5.]{Zeg96}). Nobuo Yoshida pointed out in his survey article~\cite{Yos_2} that such a decay immediately yields the uniqueness of the infinite-volume Gibbs measure (cf.~conditions (DS1), (DS2), and (DS3) in~\cite{Yos01}). Therefore there is no phase transition on the one-dimensional lattice provided the range of the interaction is finite.
\medskip

In this article, we extend the study of Zegarlinski and Yoshida to the case of infinite-range interaction. We will show that one-dimensional lattice systems of continuous spins do not have a phase transition in the same region of interaction as is known in the Ising model. Hence, we assume that the interaction decays as in~\eqref{e_decay_ising} and deduce the uniqueness of the infinite-volume Gibbs measure. \medskip

The hard ingredient of the analysis is 
to deduce that the spin-spin correlation function decays like
  \begin{equation}\label{e_decay_spin_spin_correlation}
    \frac{1}{|i-j|^{2 + \tilde \alpha}+1}
  \end{equation}
for some $0< \tilde \alpha < \alpha$. This estimate can be understood as an extension of Zegarlinski`s correlation estimate to infinite-range interaction (cf.~\cite[Lemma~4.5.]{Zeg96} and Theorem~\ref{p_mr_decay} from below). \medskip

Once the decay of correlations is established (see Theorem~\ref{p_mr_decay} from below), one can apply a recent result of one of the author (cf.~\cite[Theorem 1.14]{OR_rev}) which establishes the uniqueness of the Gibbs state (see~Theorem~\ref{p_unique_Gibbs}). We want to point out that the decay of correlations also yields a logarithmic Sobolev inequality (LSI) for the finite-volume Gibbs measure with a uniform constant in the system size and the boundary condition (cf.~\cite[Theorem 1.7]{OR_rev} and Theorem~\ref{p_mr_LSI} from below). Therefore one can say that this article generalizes Zegarlinski's main result (cf.~\cite[Theorem 4.1]{Zeg96}) from finite range to infinite-range interaction. \medskip

Now let us have a closer look at how the decay of correlations is deduced in this work. 
\begin{notation}
Let $S\subset \mathds{Z}$ be an arbitrary subset of $\mathds{Z}$. For convenience, we write $x^S$ as a shorthand for $(x_i)_{i \in S}$. 
\end{notation}
\begin{definition}[Tempered spin-values]
Given a finite subset $\Lambda\subset \mathds{Z}$, we call the spin values $x^{\mathds{Z}\backslash \Lambda}$ tempered, if for all $i \in \Lambda$
\begin{equation*}
  \sum_{j \in \mathds{Z} \backslash {\Lambda}} |M_{ij}| \ |x_j| < \infty.
\end{equation*} 
\end{definition}
\begin{definition}[Finite-volume Gibbs measure]
  Let $\Lambda$ be a finite subset of the lattice $Z$ and let $x^{\mathds{Z} \backslash \Lambda}$ be a tempered state. We call the measure $\mu_{\Lambda}( dx^{\Lambda})$ the finite-volume Gibbs measure associated to the Hamiltonian $H$ with boundary values $x^{\mathds{Z} \backslash \Lambda}$, if it is a probability measure on the space $\mathds{R}^{\Lambda}$ given by the density
\begin{equation}
  \label{e_d_Gibbs_measure}
	\mu_{\Lambda}(dx^\Lambda) = \frac{1}{Z_{\mu_\Lambda}} \e^{-H(x^\Lambda,x^{\mathds{Z} \backslash \Lambda} )} \dx x^\Lambda .
\end{equation} 
Here, $Z_{\mu_\Lambda}$ denotes the normalization constant sucht that $\mu_{\Lambda}$ is a probability measure. If there is no ambiguity, we also may write $Z$ to denote the normalization constant of a probability measure. Note that $\mu_\Lambda$ depends on the spin values $x^{\mathds{Z} \backslash \Lambda}$ outside of the set $\Lambda$.
\end{definition}
 The main result of this article is:
\begin{theorem}[Decay of spin-spin correlations]\label{p_mr_decay}
Assume that the Hamiltonian $H:\mathds{R}^{\mathds{Z}} \to \mathds{R} $ given by~\eqref{e_d_Hamiltonian} satisfies the assumptions~\eqref{e_cond_psi}~-~\eqref{e_decay_ising}.\newline
Then there exist a constant $0 < \tilde \alpha < \alpha$, where
$\alpha$ is given by~\eqref{e_decay_ising}, such that 
\begin{equation*}
  |\cov_{\mu_\Lambda } (x_i,x_j)| \lesssim \ \frac{1}{|i-j|^{2 + \tilde \alpha}+1}
\end{equation*}
uniformly in $\Lambda \subset \mathds{Z}$ and $i,j \in \Lambda$.
\end{theorem}

The proof of Theorem~\ref{p_mr_decay} is stated in Section~\ref{s_covariance_estimate}. It is motivated by the approach for the Ising model and consists of three major steps (cf.~\cite{FroeSpenc}):
\begin{itemize}
\item In the first step we show that it suffices to estimate the covariances of an associated Gibbs measure with better properties than the original Gibbs measure. More precisely, the associated Gibbs measure which will be easier to study is one whose Hamiltonian has attractive interactions and symmetric single-site potentials.
\item 
In the second step we show that the spin-spin correlation function of the associated Gibbs measure decays like (see in Theorem~\ref{p_mr_decay_reduced_measure} below)
\begin{equation}\label{e_preliminary_decay_covariances}
  \frac{1}{|i-j|^{\tilde \alpha}}.
\end{equation}
For this purpose we extend the strategy of Zegarlinski~\cite{Zeg96} from finite-range interaction to infinite-range interaction by a perturbation argument. We also improve some moment estimates of~\cite{Zeg96}, which is crucial for our argument. Note that the decay of~\eqref{e_preliminary_decay_covariances} is two orders less than claimed in Theorem~\ref{p_mr_decay}. 
\item In the third step we improve the suboptimal decay of~\eqref{e_preliminary_decay_covariances} to the desired order in Theorem~\ref{p_mr_decay} by applying an iterative scheme based on Lebowitz inequalities (cf.~\cite[Section~3]{Simon}). 
\end{itemize}

\begin{remark}\label{r_linear_term}
 Note that the structural assumptions~\eqref{e_cond_psi}~-~\eqref{e_strictly_diag_dominant} on the Hamiltonian $H$ are invariant under adding a linear term
  \begin{equation*}
    \sum_{i \in \mathds{Z}} x_ib_i
  \end{equation*}
for arbitrary $b_i \in \mathds{R}$. Therefore the constant in the decay of correlations of Theorem~\ref{p_mr_decay} is invariant under adding a linear term to the Hamiltonian. Such a linear term can be interpreted as a field acting on the system. If the coefficients $b_i$ are chosen randomly, one calls the linear term random field.
\end{remark}

Let us we turn to the uniqueness of the infinite-volume Gibbs measure. A direct consequence of the decay of correlations of Theorem~\ref{p_mr_decay} that the finite-volume Gibbs measure $\mu_{\Lambda}$ satisfies a LSI uniformly in $\Lambda$ and in the boundary values $x^{\mathds{Z} \backslash \Lambda}$ i.e
\begin{theorem}(Uniform logarithmic Sobolev inequality) \label{p_mr_LSI}
  We assume that the formal Hamiltonian $H:\mathds{R}^{\mathds{Z}} \to \mathds{R} $ given by~\eqref{e_d_Hamiltonian} satisfies the Assumptions~\eqref{e_cond_psi}~-~\eqref{e_decay_ising}.\newline
Then the finite-volume Gibbs measure satisfies a LSI uniformly in $\Lambda$ and $x^{\mathds{Z} \backslash \Lambda}$ i.e.~for all functions $f \geq 0$ 
  \begin{equation}\label{e_definition_of_LSI}
  \int f \log f \ d \mu - \int f  d\mu  \log \left( \int f  d\mu \right) \leq \frac{1}{2 \varrho} \int \frac{|\nabla f|^2}{f} d\mu,  
  \end{equation}
where the constant $\varrho>0$ is independent of $\Lambda$ and $x^{\mathds{Z} \backslash \Lambda}$.
\end{theorem}
The proof of Theorem~\ref{p_mr_LSI} consists of a direct application of~\cite[Theorem 1.7]{OR_rev} and therefore is omitted in this article. The fact that a uniform LSI in combination with decay of correlation yields the uniqueness of the Gibbs state is already known from the case of finite-range interaction (cf.~\cite{Roy07},~\cite{Zitt}, and~\cite{Yos01}). Therefore it is not surprising that one can generalize this statement to infinite-range interaction (cf.~\cite[Theorem 1.14]{OR_rev}). Hence, a direct application of \cite[Theorem 1.14]{OR_rev} yields the uniqueness of the Gibbs measure:
\begin{definition}[Infinite-Volume Gibbs measure]
  Let $\mu$ be a probability measure on the state space $\mathds{R}^{\mathds{Z}}$ equipped with the standart product Borel sigma-algebra. For any finite subset $\Lambda \subset \mathds{Z}$ we decompose the measure $\mu$ into the conditional measure $\mu(dx^\Lambda| x^{\mathds{Z} \backslash \Lambda})$ and the marginal $\bar \mu (d x^{\mathds{Z} \backslash \Lambda})$. This means that for any test function $f$ it holds
\begin{equation*}
  \int f(x) \mu (dx) = \int \int f(x) \mu(dx^\Lambda| x^{\mathds{Z} \backslash \Lambda}) \bar \mu (d x^{\mathds{Z} \backslash \Lambda}).
\end{equation*}
We say that the measure $\mu$ is the infinite-volume Gibbs measure associated to the Hamiltonian $H$, if the conditional measures $\mu(dx^\Lambda| x^{\mathds{Z} \backslash \Lambda})$ are given by the finite-volume Gibbs measures $\mu_{\Lambda}(dx^\Lambda)$ defined by~\eqref{e_d_Gibbs_measure} i.e.
\begin{equation*}
  \mu(dx^\Lambda| x^{\mathds{Z} \backslash \Lambda}) = \mu_\Lambda (dx^\Lambda).
\end{equation*} 
The equations of the last identity are also called Dobrushin-Lanford-Ruelle (DLR) equations. 
\end{definition}
\begin{theorem}[Uniqueness of the infinite-volume Gibbs measure]\label{p_unique_Gibbs}
 We assume that the formal Hamiltonian $H:\mathds{R}^{\mathds{Z}} \to \mathds{R} $ given by~\eqref{e_d_Hamiltonian} satisfies the Assumptions~\eqref{e_cond_psi}~-~\eqref{e_decay_ising}.\newline
Then there is at most one unique Gibbs measure $\mu$ associated to the Hamiltonian $H$ satisfying the uniform bound
\begin{equation}~\label{e_sup_moment}
  \sup_{i \in \mathds{Z}} \var_{\mu}(x_i) < \infty.
\end{equation}
\end{theorem}

The variance condition~\ref{e_sup_moment} is kind of standart. For example it follows from a widely used moment condition
\begin{equation*}
  \sup_{i \in \mathds{Z}} \int (x_i)^2 \mu (dx) < \infty,
\end{equation*}
which is used in the study of infinite-volume Gibbs measures (see for example~~\cite{BHK82} and~\cite[Chapter~4]{Roy07}). It is relatively easy to show that the condition~\eqref{e_sup_moment} is invariant under adding a bounded random field to the Hamiltonian~$H$ (cf.~Remark~\ref{r_linear_term}).

 
\begin{remark}
  In this article, we do not show the existence of an infinite-volume Gibbs measure. However, the authors of this article believe that under the assumption~\eqref{e_sup_moment} the existence should follow by an compactness argument similarly to the one used in~\cite{BHK82}.   
\end{remark}

\begin{remark}
Usually one considers finite-volume Gibbs measure with some inverse temperature $\beta >0$ i.e.
  \begin{equation*}
    	\mu_\Lambda (dx^\Lambda ) = \frac{1}{Z_\mu} \e^{- \beta H(x^\Lambda, x^{\mathds{Z} \backslash \Lambda})} \dx x \qquad \mbox{for } x^\Lambda \in \mathds{R}^{\Lambda}.
  \end{equation*}
This case is also contained in the main result of this article, because the Hamiltonian $\beta H$ still satisfies the structural Assumptions~\eqref{e_cond_psi}~-~\eqref{e_decay_ising}. Of course, the constants in the decay of correlations result of Theorem~\ref{p_mr_decay} will depend on the inverse temperature~$\beta$. 
\end{remark}

\begin{remark}
  Because we assume that the matrix $M= (M_{ij})$ is strictly diagonal dominant (cf.~\eqref{e_strictly_diag_dominant}), the full single-site potential 
  \begin{equation*}
    \psi_i(x_i) + m_{ii} x_i^2 = m_{ii} x_i^2 + \psi_i^c(x_i) + \psi_i^b (x_i)
  \end{equation*}
is perturbed strictly-convex. This is the same structural assumption as used in the article~\cite{MO}, which seems to be vey natural for unbounded spin systems.
\end{remark}

In order to avoid confusion, let us make the notation $a \lesssim b$ more precise. 
\begin{definition}\label{def:dep}
	We will use the notation $a \lesssim b$ for quantities $a$ and $b$
	to indicate that there is a constant $C \ge 0$
	which depends only on the lower bound~$\delta$ and upper bounds for $|\psi_i^b|$, $|(\psi_i^b)'|$, and $\sum_{i,j \in \mathds{Z}} |M_{ij}|$ such that $a \le C b$. In the same manner, if we assert the existence of certain constants, they may freely depend on the above mentioned quantities, whereas all other dependencies will be pointed out.
\end{definition}

We close the introduction by giving an outline over the article.\smallskip
\begin{itemize}
\item  In Section~\ref{s_covariance_estimate}, we outline the strategy of the proof of Theorem~\ref{p_mr_decay}. 
\item In Section~\ref{s_griffiths_estimates}, we carry out the first step of the proof of Theorem~\ref{p_mr_decay}. We show that it suffices to estimate the covariances with respect to a ferromagnetic finite-volume Gibbs measure with nice symmetries.
\item In Section~\ref{s_zegar_perturbation_argument}, we carry out the second step of the proof of Theorem~\ref{p_mr_decay} deducing a preliminary decay of correlations by perturbing Zegarlinski's argument.
\item In Section~\ref{s_cov_est_postprocess}, we carry out the last step of the proof of Theorem~\ref{p_mr_decay} applying an iterative scheme based on Lebowitz inequalities to improve the preliminary decay of correlations to the correct order.
\end{itemize}

\section{Proof of Theorem~\ref{p_mr_decay}: Outline of the argument}\label{s_covariance_estimate}

Before we turn to the covariance estimate, we apply a transformation that allows us to analyze covariances of an associated measure $\mu_{\Lambda,q}$, which has better properties than the original measure $\mu_{\Lambda}$. This standart procedure is, for example, also applied in~\cite{Zeg96}. \medskip

We represent the covariance $\cov_{\mu_{\Lambda}}(x_i, x_j)$ in the following way. From the definition of the covariance it follows that
\begin{equation*}
  \cov_{\mu_{\Lambda}}(x_i, x_j) = \frac{1}{2} \int \int (x_i -y_i) (x_j -y_j) \mu_{\Lambda}(dx) \mu_{\Lambda} (dy).
\end{equation*}
By the change of coordinates $x_k= q_k + p_k$ and $y_k= q_k -p_k$ for all $k \in \Lambda$, the last identity yields by using the definition~\eqref{e_d_Gibbs_measure} of the finite-volume Gibbs measure $\mu_{\Lambda}$ that
\begin{align*}
  \cov_{\mu_{\Lambda}}(x_i, x_j) &= 2 \int \int p_i  p_j \  \underbrace{\frac{e^{-H(q^{\Lambda}+p^{\Lambda}, x^{\mathds{Z}\backslash \Lambda}) - H(q^{\Lambda}-p^{\Lambda}, x^{\mathds{Z}\backslash \Lambda})}}{\int e^{-H(q^{\Lambda}+p^{\Lambda}, x^{\mathds{Z}\backslash \Lambda})  -H(q^{\Lambda}-p^{\Lambda}, x^{\mathds{Z}\backslash \Lambda})}  \dx p^{\Lambda} \dx q^{\Lambda} }   \dx p^{\Lambda} \dx q^{\Lambda} }_{=: d \tilde \mu_{\Lambda} (q^{\Lambda},p^{\Lambda})} . 
\end{align*}
By conditioning on the values $q^{\Lambda}$ it follows from the definition~\eqref{e_d_Hamiltonian} of~$H$ that
\begin{equation} \label{e:repre_covariance}
  \cov_{\mu_{\Lambda}}(x_i, x_j) = 2 \mathds{E}_{\tilde \mu_{\Lambda}} \left[ \mathds{E}_{\mu_{\Lambda,q}} \left[ p_i  p_j \right] \right].
\end{equation}
Here, the conditional measure $\mu_{\Lambda,q}$ is given by the density
\begin{equation}
  \label{eq:def_mu_q}
	\mu_{\Lambda,q}(dp^\Lambda) \coloneqq \frac{1}{Z_{\mu_{\Lambda,q}}} \e^{-\sum_{k \in \Lambda} \psi_{k,q}(p_k) - \sum_{k,l \in \Lambda} M_{kl} p_k p_l} \dx p^\Lambda  
\end{equation}
with single-site potentials $\psi_{k,q} \coloneqq \psi_{k,q}^c + \psi_{k,q}^b$ defined by 
\begin{align*}
   \psi_{k,q}^c(p_k) & \coloneqq \psi_k^c(q_k + p_k) + \psi_k^c(q_k - p_k) \qquad \mbox{and} \\
  \psi_{k,q}^b(p_k) & \coloneqq \psi_k^b(q_k + p_k) + \psi_k^b(q_k - p_k).  
\end{align*}
The conditional measure $\mu_{\Lambda,q}$ has three properties: the single-potentials $\psi_{k,q}$ are symmetric i.e.~$\psi_{k,q}(p_k)=\psi_{k,q}(-p_k)$, therefore the variables $p^{\Lambda}$ are centered, and $\mu_{\Lambda,q}$  is independent of the fixed spin values $x^{\mathds{Z} \backslash \Lambda}$. Therefore, it holds 
\begin{equation}\label{e_repr_covariance_moment}
  \int p_i \dx \mu_{\Lambda, q} (p^{\Lambda})=0 \qquad \mbox{and thus} \qquad \mathds{E}_{\mu_{\Lambda,q}}  \left[ p_i  p_j \right] = \cov_{\mu_{\Lambda, q}} (p_i, p_j).
\end{equation}
In view of~\eqref{e:repre_covariance}, it therefore suffices to show that the covariances of the measure $\mu_{\Lambda,q}$ decay sufficiently fast. \medskip

In order to simplify further the structure of the associated Gibbs measure~$\mu_{\Lambda,q}$, the next statement shows that we may assume w.l.o.g.~that the interactions in the Hamiltonian are attractive i.e.

\begin{lemma}\label{p:attractive_interact_dominates}
  Assume that the finite-volume Gibbs measure~$\mu_{\Lambda,q}$ is given by (\ref{eq:def_mu_q}). Additionally, consider the corresponding finite-volume Gibbs measure~~$\mu_{\Lambda,q,|M|}$ with attractive interaction i.e. 
\begin{equation}
  \label{eq:def_mu_q_|M|}
	\mu_{\Lambda,q,|M|}(dp^\Lambda) \coloneqq \frac{1}{Z_{\mu_{\Lambda,q,|M|}}} \e^{-\sum_{k \in \Lambda} \psi_{k,q}(p_k) - \sum_{k,l \in \Lambda} |M_{kl}| p_k p_l} \dx p^\Lambda  
\end{equation}
Then it holds that for any $i,j \in \Lambda$
\begin{equation}
  \label{eq:covariance_domination}
  | \cov_{\mu_{\Lambda, q}} (p_i,p_j)  | \leq \cov_{\mu_{\Lambda, q, |M|}} (p_i,p_j) .
\end{equation}
\end{lemma}
The content of Lemma~\ref{p:attractive_interact_dominates} is also standart (see for example~\cite[Section~3]{Simon}). However, for the sake of completeness we state the proof of Lemma~\ref{p:attractive_interact_dominates} in Section~\ref{s_griffiths_estimates}. One can apply the same argument as used in~\cite{HorMor} for discrete spins, as soon as one shows that the second GKS inequality~\cite{KellySher} holds in our situation. \medskip

As announced in the introduction, we derive the following covariance estimate in the second step.
\begin{proposition}[Preliminary decay of spin-spin correlations] \label{p_mr_decay_reduced_measure}
Assume that the formal Hamiltonian $H:\mathds{R}^{\mathds{Z}} \to \mathds{R} $ given by~\eqref{e_d_Hamiltonian} satisfies the assumptions~\eqref{e_cond_psi}~-~\eqref{e_decay_ising}. We consider the conditional measure $\mu_{\Lambda,q,|M|}$ given by~(\ref{eq:def_mu_q_|M|}). 
Then there exist a constant $0 < \tilde \alpha < \alpha$ such that 
\begin{equation}\label{e_algeb_decay_sp_sp_corr}
  \cov_{\mu_\Lambda,q,|M| } (x_i,x_j) \lesssim \ \frac{1}{|i-j|^{\tilde \alpha}}
\end{equation}
uniformly in $\Lambda \subset \mathds{Z}$, $i,j \in \Lambda$, and $p^\Lambda $.
\end{proposition}
The fact that Proposition~\ref{p_mr_decay_reduced_measure} does not show the right order of decay is not surprising: The reason is that Proposition~\ref{p_mr_decay_reduced_measure} is derived by a combination of the approach of Zegarlinski~\cite[Theorem~4.1.]{Zeg96} for finite-range interaction and a perturbation argument. The proof of Proposition~\ref{p_mr_decay_reduced_measure} is stated in Subsection~\ref{s_zegar_perturbation_argument}. In the perturbation argument, we pass from the original Hamiltonian $H$ to a Hamiltonian with suitably truncated interactions. We also have to quantize the estimates more precisely than it was needed in Zegarlinski's application. This quantification is achieved by using slightly improved moment estimates compared to~\cite{Zeg96} (cf.~notes before Lemma~\ref{lem:zeg} from below). At this point one fundamentally uses the simplified structure of the measure~$\mu_{\Lambda,q,|M|}$, namely that the convex part of the single-site potentials $\psi_{k,q}^c$ has a global minimum in $0$.  \medskip

After deducing Proposition~\ref{p_mr_decay_reduced_measure}, we improve the suboptimal covariance estimate~\eqref{e_algeb_decay_sp_sp_corr} by postprocessing the estimate recursively, which was already a successful strategy in the case of discrete spins (cf.~\cite{FroeSpenc}). In order to get a recursive relation for the covariance, we apply that the measure $\mu_{\Lambda,q,|M|}$ obeys the Lebowitz inequality~\cite[Section~3]{Simon}. We carry out the details of this argument in Subsection~\ref{s_cov_est_postprocess}, where the proof of Theorem~\ref{p_mr_decay} is also stated.   \medskip

\section{Domination of covariances by ferromagnetic systems: proof of Lemma~\ref{p:attractive_interact_dominates}}\label{s_griffiths_estimates}

In the argument for Lemma~\ref{p:attractive_interact_dominates}, we follow the lines of the proof of~\cite[Theorem~1]{HorMor}.
\begin{proof}[Proof of Lemma~\ref{p:attractive_interact_dominates}]
  Firstly, observe that for any $i \in \Lambda$
  \begin{equation*}
    \int p_i \mu_{\Lambda, q} = \int p_i \mu_{\Lambda, q, |M|}= 0   
  \end{equation*}
   by symmetry of the Hamiltonian w.r.t.~the measures~$\mu_{\Lambda, q}$ and~$\mu_{\Lambda, q, |M|}$. We start with showing
  \begin{equation} \label{e:ferromangnetic_dominates_upper_bound}
     \cov_{\mu_{\Lambda, q}} (p_i,p_j) =  \int p_i p_j \dx \mu_{\Lambda, q}  \leq \cov_{\mu_{\Lambda, q, |M|}} (p_i,p_j)= \int p_i p_j \dx \mu_{\Lambda, q, |M|}.
  \end{equation}
  For this purpose, we introduce the auxiliary Hamiltonian $\tilde H (x^\Lambda, \sigma)$ by
  \begin{equation}\label{e:aux_Hamil_ferro_domination}
    \tilde H(p^\Lambda, \sigma) = \sum_{k \in \Lambda} \psi_{k,q} (p_k) + \sum_{k,l \notin \Omega} M_{kl} p_k p_l  - \sum_{k,l \in \Omega} M_{kl} p_k p_l \sigma,
  \end{equation}
  where the set $\Omega$ is defined by
  \begin{equation*}
    \Omega \coloneqq \left\{ \left\{k,l \right\} \subset \Lambda \ : \ M_{kl} > 0 \right\}
  \end{equation*}
  and $\sigma \in  \left\{ -1, 1 \right\}$ is a ghost Ising spin on a ghost site. We introduce the finite-volume Gibbs measure associated to the Hamiltonian $\tilde H$ by
  \begin{equation}
    \label{eq:def_aux_meas_ferro_domination}
    \tilde \mu (dp^\Lambda,d \sigma)= \frac{1}{Z_{\tilde \mu}} e^{-\tilde H(p^\Lambda, \sigma)} \dx p^\Lambda \  \dx \gamma(\sigma), 
  \end{equation}
where $\gamma$ denotes the measure 
$\gamma = \frac{1}{2} (\delta_{-1} + \delta_{1})$ corresponding to the ghost Ising spin $\sigma$. Because the Hamiltonian $\tilde H(p^\Lambda, \sigma)$ only has ferromagnetic interaction the following special case of the second GKS inequality holds (for a proof see Lemma~\ref{p_second_GKS} from below)
\begin{equation}\label{eq:ferro_domination_second_GKS}
  \int p_i p_j \sigma \dx \tilde \mu - \int p_i p_j \dx \tilde \mu \int \sigma \dx \tilde \mu  \geq 0.
\end{equation}
Note that $\tilde H(p^\Lambda, -1)$ coincides with the Hamiltonian of the measure $\mu_{\Lambda, q}$, whereas $\tilde H(p^\Lambda, 1)$ coincides with the Hamiltonian of the measure~$\mu_{\Lambda, q, |M|}$. Therefore we get the identity 
\begin{equation*}
  \int \sigma \dx \tilde \mu  = \frac{1}{Z_{\tilde \mu}} \left( Z_{\mu_{\Lambda, q, |M|}} - Z_{\mu_{\Lambda, q}}  \right) \qquad \mbox{and} \qquad Z_{\tilde \mu} = Z_{\mu_{\Lambda, q, |M|}} + Z_{\mu_{\Lambda, q}}.   
\end{equation*}
Hence, multiplying the estimate~\eqref{eq:ferro_domination_second_GKS} with $Z_{\tilde \mu}^2$ yields
\begin{align*}
 & \left(  \int p_i p_j e^{\tilde H(p^\Lambda, 1)} \dx p^\Lambda - \int p_i p_j e^{\tilde H(p^\Lambda, -1)} \dx p^\Lambda  \right)   \left(  Z_{\mu_{\Lambda, q, |M|}} + Z_{\mu_{\Lambda, q}} \right) \\
& \qquad \geq  \left(  \int p_i p_j e^{\tilde H(p^\Lambda, 1)} \dx p^\Lambda + \int p_i p_j e^{\tilde H(p^\Lambda, -1)} \dx p^\Lambda  \right) \left(   Z_{\mu_{\Lambda, q, |M|}} -Z_{\mu_{\Lambda, q}} \right),
\end{align*}
from which the desired estimate~\eqref{e:ferromangnetic_dominates_upper_bound} follows. \newline
It is left to show that
\begin{equation}\label{e:ferromangnetic_dominates_lower_bound}
  - \int p_i p_j \dx \mu_{\Lambda, q}  \leq  \int p_i p_j \dx \mu_{\Lambda, q, |M|}.
\end{equation}
For this purpose we define the auxiliary Hamiltonian $H_i$ by
\begin{equation*}
  \tilde H_i(p^\Lambda, \sigma) = \sum_{k \in \Lambda} \psi_{k,q} (p_k) + \sum_{\substack{k,l \in \Lambda\\ k \neq i, l \neq i } } M_{kl} p_k p_l  - \sum_{\substack{k,l \in \Lambda\\ k =i \vee l=i} } M_{kl} p_k p_l \sigma.
\end{equation*}
Note that this Hamiltonian flipped the interaction w.r.t.~the $i$-th site. Considering the finite-volume Gibbs measure $\mu_i (dp^\Lambda)$ that is associated to the Hamiltonian $H_i$ we get by substitution
\begin{equation*}
  - \int p_ip_j \dx \mu_{\Lambda, q} = \int p_i p_j \dx \mu_i.
\end{equation*}
Hence, the desired estimate~\eqref{e:ferromangnetic_dominates_lower_bound} follows from the estimate 
\begin{equation*}
  \int p_i p_j \dx \mu_i \leq \int p_i p_j \dx \mu_{\Lambda, q, |M|},
\end{equation*}
which can be shown in the same way as~\eqref{e:ferromangnetic_dominates_upper_bound}.
\end{proof}
In the proof of Lemma~\ref{p:attractive_interact_dominates} we used the following fact.
\begin{lemma}\label{p_second_GKS}
  The measure $\tilde \mu (dp^\Lambda,d \sigma)$ given by~\eqref{eq:def_aux_meas_ferro_domination} satisfies for any $i,j \in \Lambda$ the estimate 
  \begin{equation}\label{eq:ferro_domination_second_GKS_lemm}
      \cov_{\tilde \mu} (p_ip_j, \sigma) = \int p_i p_j \sigma \dx \tilde \mu - \int p_i p_j \dx \tilde \mu \int \sigma \dx \tilde \mu  \geq 0.
  \end{equation}
\end{lemma}
The estimate~\eqref{eq:ferro_domination_second_GKS_lemm} is just a special case of the second GKS inequality~\cite{KellySher}. To be self contained, we state a proof of Lemma~\ref{p_second_GKS} which uses the idea of Sylvester of expanding the exponential function~\cite{Sylvester}
.
\begin{proof}[Proof of Lemma~\ref{p_second_GKS}]
  By doubling the variables, we represent the covariance in the following way (cf.~begin of Section~\ref{s_covariance_estimate})
  \begin{equation*}
    \cov_{\tilde \mu} (p_ip_j, \sigma) = \frac{1}{2}\int \int (p_ip_j - \tilde p_i \tilde p_j)  (\sigma  - \tilde \sigma) \tilde \mu (d p^{\Lambda} , d \sigma) \tilde \mu (d \tilde p^{\Lambda}, d \tilde \sigma).
  \end{equation*}
  Applying the the transformation $p_k= \hat p_k + \hat q_k$ and $\tilde p_k= -\hat p_k + \hat q_k$ for all $k \in \Lambda$ and the transformation $\sigma = \alpha + \beta$ and $\tilde \sigma = -\alpha + \beta$  yields the identity
  \begin{align}
  \notag & \cov_{\tilde \mu} (p_ip_j, \sigma) = \\
   &\frac{1}{Z} \underbrace{\int \int ( \hat p_i \hat q_j + \hat p_j \hat q_i ) \alpha e^{-\tilde H (\hat p^\Lambda + \hat q^{\Lambda}, \alpha + \beta) -\tilde H (-\hat p^\Lambda + \hat q^{\Lambda}, -\alpha + \beta)} \dx \hat p^\Lambda \dx \hat q^\Lambda  \ \nu (d\alpha) \ \nu (d \beta)}_{=:T}, \label{rep_cov_second_GKS}
  \end{align}
where $Z>0$ is a unspecified normalization constant and the measure $\nu$ on the space $\left\{ -1,0,1 \right\}$ is given by
\begin{align*}
  \nu(-1) = \nu(1) = \frac{1}{4} \qquad \mbox{and} \qquad \nu(0) =\frac{1}{2}.
\end{align*}
For the sign of the covariance $\cov_{\tilde \mu} (p_ip_j, \sigma)$ only the integral term on the right hand side of~\eqref{rep_cov_second_GKS} is important. \newline
Hence, let us have a closer look at the integral term. Note that the Hamiltonian $\tilde H$ can be written as 
\begin{align*}
   \tilde H (p,q,\sigma)  = \tilde I (p,q,\sigma)  + \tilde P (p,q,\sigma)  
\end{align*}
where $\tilde I$ denotes the ferromagnetic interaction 
\begin{align*}
  \tilde I (p,q,\sigma) =  \sum_{\substack{k,l \in \Lambda\\ k \neq i, l \neq i } } M_{kl} p_k p_l  - \sum_{\substack{k,l \in \Lambda\\ k =i \vee l=i} } M_{kl} p_k p_l \sigma
\end{align*}
and $\tilde P$ denotes the single-site potentials
\begin{align*}
   \tilde P (p,q,\sigma) =  \sum_{k \in \Lambda} \psi_{k,q} (p_k).
\end{align*}
  Recall that the single-site potential $\psi_{k,q} : \mathbb{R} \to \mathbb{R}$ is an even function, therefore the function $\tilde P (p,q,\sigma)$ is also even in any coordinate.\newline
 Using this decomposition of $\tilde H$ we can write the integral term of~\eqref{rep_cov_second_GKS} as
  \begin{align*}
&    T = \int \int ( \hat p_i \hat q_j + \hat p_j \hat q_i ) \alpha e^{-\tilde I (\hat p^\Lambda + \hat q^{\Lambda}, \alpha + \beta) -\tilde I (-\hat p^\Lambda + \hat q^{\Lambda}, -\alpha + \beta)} \\
& \qquad \qquad \times e^{-\tilde P (\hat p^\Lambda + \hat q^{\Lambda}, \alpha + \beta) -\tilde P (-\hat p^\Lambda + \hat q^{\Lambda}, -\alpha + \beta)} \dx \hat p^\Lambda \dx \hat q^\Lambda  \ \nu (d\alpha) \ \nu (d \beta).
  \end{align*}
Expanding the interaction part into a Taylor series yields that
  \begin{align}
&    T = \int \int p(\hat p, \hat q, \alpha, \beta)  e^{-\tilde P (\hat p^\Lambda + \hat q^{\Lambda}, \alpha + \beta) -\tilde P (-\hat p^\Lambda + \hat q^{\Lambda}, -\alpha + \beta)} \dx \hat p^\Lambda \dx \hat q^\Lambda  \ \nu (d\alpha)  \nu (d \beta), \label{e_rep_cov_poly}
  \end{align}
where $p(\hat p, \hat q, \alpha, \beta)$ is a polynomial (formally of infinite degree). The coefficients of the polynomial $p$ are nonnegative due to the fact that the interaction term $\tilde I$ of the Hamiltonian $\tilde H$ is ferromagnetic. \newline
Because the function 
\begin{align*}
  \tilde P (\hat p^\Lambda + \hat q^{\Lambda}, \alpha + \beta) +\tilde P (-\hat p^\Lambda + \hat q^{\Lambda}, -\alpha + \beta)
\end{align*}
is even in any coordinate of $\hat p^\Lambda,\hat q^{\Lambda}, \alpha,$ and $\beta$, it follows that
\begin{align*}
 & \int \int \hat p_i^{n_i} \hat q_j^{n_j} \alpha^{n_\alpha} \beta^{n_\beta}  e^{-\tilde P (\hat p^\Lambda + \hat q^{\Lambda}, \alpha + \beta) -\tilde P (-\hat p^\Lambda + \hat q^{\Lambda}, -\alpha + \beta)} \dx \hat p^\Lambda \dx \hat q^\Lambda  \ \nu (d\alpha)  \nu (d \beta)  \\ 
& \qquad=0,
\end{align*}
if any number $n_i, n_j, {n_\alpha}$, and $n_\beta$ is odd. Therefore, only terms of even order remain in the representation~\eqref{e_rep_cov_poly} from above, which yields the desired estimate 
\begin{align*}
  \cov_{\tilde \mu} (p_ip_j, \sigma) = \frac{1}{Z} T \geq 0.
\end{align*}
\end{proof}

\section{A perturbation of Zegarlinski's approach: proof of Proposition~\ref{p_mr_decay_reduced_measure}}\label{s_zegar_perturbation_argument}

For the proof of Proposition~\ref{p_mr_decay_reduced_measure} we follow the approach of Zegarlinski~\cite{Zeg96}. In the first step, we provide some auxiliary moment estimates. These estimates represent the fundament of the proof of Proposition~\ref{p_mr_decay_reduced_measure}. We start with the following standart tool on stochastic domination of one-dimensional measures with respect to monotone perturbations.

\begin{lemma}\label{lem:FKG}
	Let $\nu$ be a probability measure on $\mathds{R}$ and let $f$ and $\psi$ be monotone functions
	on the support of $\nu$ with the same direction of monotonicity. Then
	\[
		\frac{\int f \e^{-\psi} \dx\nu}{\int \e^{-\psi} \dx\nu} \le \int f \, \dx\nu,
	\]
	provided these integrals exist.
\end{lemma}
\begin{proof}[Proof of Lemma~\ref{lem:FKG}]
	Consider the probability measures $\nu_\lambda$ given by 
        \begin{equation*}
          \nu_\lambda(dx) \coloneqq \frac{\e^{-\lambda \psi(x)} \dx\nu(x)}{\int \e^{-\lambda \psi} \dx\nu}.
        \end{equation*}
	Then direct calculation yields the estimate
        \begin{align*}
        	\frac{\dx}{\dx \lambda} \int f \, \dx\nu_\lambda
			& = -\cov_{\nu_\lambda}(f,\psi) \\
			& = -\int \int (f(x)-f(y)) (\psi(x)-\psi(y)) \dx\nu_\lambda(x) \dx\nu_\lambda(y) \le 0  
        \end{align*}
	by monotonicity. Hence, the desired estimate follows from an application of the fundamental theorem of calculs i.e.
	\[
		\frac{\int f \e^{-\psi} \dx\nu}{\int \e^{-\psi} \dx\nu} - \int f \, \dx\nu
			= \int f \, \dx\nu_1 - \int f \, \dx\nu_0
			= \int_0^1 \frac{\dx}{\dx \lambda} \int f \, \dx\nu_\lambda \; \dx\lambda \le 0.
	\qedhere
	\]
\end{proof}

The Lemma~\ref{lem:FKG} is used to deduce the following control of the exponential moment of the finite-volume Gibbs measure $\mu_{\Lambda}$. 

\begin{lemma}\label{lem:moments}
  We assume that the formal Hamiltonian $H:\mathds{R}^{\mathds{Z}} \to \mathds{R} $ given by
   \begin{equation*}
	H(x) = \sum_{i \in \mathds{Z} } \psi_i(x_i) +  \sum_{i,j \in \mathds{Z}} M_{ij} x_i x_j.  
\end{equation*}
satisfies the Assumptions~\eqref{e_cond_psi}~-~\eqref{e_decay_ising}.\newline
Additionally, we assume that for all $ i \in \mathds{Z}$ the convex part $\psi_i^c$ of the single-site potentials $\psi_i$ has a global minimum in $x_i=0$. \newline
Let $\delta >0$ be given by~\eqref{e_strictly_diag_dominant}. Then for every $0 \le a \le \frac{\delta}{2}$ and any subset $\Lambda \subset \mathds{Z}$ it holds
        \begin{equation} \label{e:exponential_moment}
          \E_{\mu_{\Lambda}} \bigl[\e^{a p_i^2}\bigr] \lesssim 1.
        \end{equation}
In particular, for any $k \in \mathds{N}_0$ this yields 
\begin{equation} \label{e:arbitrary_moment}
  \E_{\mu_{\Lambda}}[p_i^{2k}] \lesssim k!.
\end{equation}
\end{lemma}
The statement of Lemma~\ref{lem:moments} is a slight improvement of~\cite[Section~3]{BHK82}, because our assumptions are slightly weaker compared to~\cite {BHK82}. More precisely, $\psi_i''$ may change sign outside every compact set and there is no condition on the signs of the interaction. Moreover, our moment estimate is stronger, which is essential for our application. Still, the proof remains essentially the same.
\begin{proof}[Proof of Lemma~\ref{lem:moments}]
Let us introduce the auxiliary variables 
$$s_i \coloneqq - \frac{1}{2} \sum_{j \in \mathds{Z}, j \neq i} M_{ij} x_j \qquad \mbox{and} \qquad m_i:=M_{ii} - \frac{\delta}{4}.$$
In the first step of the proof, we will deduce the following moment estimate on the finite-volume Gibbs measure $\mu_{\left\{i \right\}}$ associated to a single-site $i \in \mathds{Z}$:
\begin{equation}
  \label{e_moment_estimate_single-site}
  \E_{\mu_{\left\{ i \right\}}} [e^{a x_i^2}] \lesssim \e^{\frac{as_i^2}{m_i(m_i-a)}} .
\end{equation}
Indeed, elementary estimates show that for every $K > 0$ it holds
	\begin{align*}
		\E_{\mu_{\left\{ i \right\} }} [\e^{a x_i^2}]
			\le \e^{aK^2} & + \frac{\int_{\lbrace x_i \le -K \rbrace} \e^{a x_i^2} \e^{-\psi_{i}(x_i) - M_{ii} x_i^2 + 2s_i x_i} \dx x_i}{\int_{\lbrace x_i \le -K \rbrace} \e^{-\psi_{i}(x_i) - M_{ii}x_i^2 + 2s_i x_i} \dx x_i} \\
				& + \frac{\int_{\lbrace x_i \ge K \rbrace} \e^{a x_i^2} \e^{-\psi_{i}(x_i) - M_{ii} x_i^2 + 2s_i x_i} \dx x_i}{\int_{\lbrace x_i \ge K \rbrace} \e^{-\psi_{i}(x_i) - M_{ii}x_i^2 + 2s_i x_i} \dx x_i}.
	\end{align*}
	We set $K \coloneqq \frac{4}{\delta} \sup_i \|\psi_i^b\|_{\mathrm{C}^1} \lesssim 1$. Then it follows by using the fact  that $\psi_{i}^c$ is convex with a global minimum at $x_i = 0$ that for $x_i \ge K$ 
	\[
		\frac{\dx}{\dx x_i} \bigl( \psi_{i}(x_i) + \frac{\delta}{4} x_i^2 \bigr) = (\psi_{i}^c)'(x_i) + (\psi_{i}^b)'(x_i) + \frac{\delta}{2} x_i \ge (\psi_{i}^b)'(x_i) + \frac{\delta}{2} x_i \ge 0
	\]
	and similarly for $x_i \le -K$
	\[
		\frac{\dx}{\dx x_i} \bigl( \psi_{i}(x_i) + \frac{\delta}{4} x_i^2 \bigr) = (\psi_{i}^c)'(x_i) + (\psi_{i}^b)'(x_i) + \frac{\delta}{2} x_i \le (\psi_{i}^b)'(x_i) + \frac{\delta}{2} x_i \le 0
	\]
	Hence, we conclude from Lemma~\ref{lem:FKG} that
	\begin{align}
		& \E_{\mu_{\left\{i \right\}} }[\e^{a x_i^2}] \notag \\
			& \le \e^{aK^2} + \frac{\int_{\lbrace x_i \le -K \rbrace} \e^{-(M_{ii}-\frac{\delta}{4}-a) x_i^2 + 2s_ix_i} \dx x_i}{\int_{\lbrace x_i \le -K \rbrace} \e^{-(M_{ii}-\frac{\delta}{4})x_i^2 + 2s_ix_i} \dx x_i}
				+ \frac{\int_{\lbrace x_i \ge K \rbrace} \e^{-(M_{ii}-\frac{\delta}{4}-a) x_i^2 + 2s_ix_i} \dx x_i}{\int_{\lbrace x_i \ge K \rbrace} \e^{-(M_{ii}-\frac{\delta}{4})x_i^2 + 2s_ix_i} \dx x_i} \notag \\
			& = \e^{aK^2} + \e^{\frac{as_i^2}{m_i(m_i-a)}} \sqrt{\frac{m_i}{m_i-a}} \biggl( \frac{\Phi\bigl(\frac{\sqrt{2}(s_i-K(m_i-a))}{\sqrt{m_i-a}}\bigr)}{\Phi\bigl(\frac{\sqrt{2}(s_i-Km_i)}{\sqrt{m_i}}\bigr)} + \frac{\Phi\bigl(\frac{-\sqrt{2}(s_i+K(m_i-a))}{\sqrt{m_i-a}}\bigr)}{\Phi\bigl(\frac{-\sqrt{2}(s_i+Km_i)}{\sqrt{m_i}}\bigr)} \biggr). \label{e:second_moment_single_site_core}
	\end{align}
	Here, $\Phi$ denotes the cumulative normal distribution function. If $s_i \le -K \sqrt{m_i}\sqrt{m_i-a}$, then
        \begin{align*}
          \frac{\sqrt{2}(s_i-K(m_i-a))}{\sqrt{m_i-a}} \leq  \frac{\sqrt{2}(s_i-Km_i)}{\sqrt{m_i}}
        \end{align*}
        by straightforward calculation. Therefore, we have
        \begin{equation}
          \label{e:second_moment_normal_distribution_first}
            		\Phi\Bigl(\frac{\sqrt{2}(s_i-K(m_i-a))}{\sqrt{m_i-a}}\Bigr) \le \Phi\Bigl(\frac{\sqrt{2}(s_i-Km_i)}{\sqrt{m_i}}\Bigr)
        \end{equation}
	by monotonicity. If $s_i \ge -K \sqrt{m_i} \sqrt{m_i-a}$, then
        \begin{align}
        	\Phi\Bigl(\frac{\sqrt{2}(s_i-Km_i)}{\sqrt{m_i}}\Bigr)
			& \ge \Phi\Bigl(-\sqrt{2}K(\sqrt{m_i-a} + \sqrt{m_i})\Bigr) \notag \\
			& \ge \Phi\Bigl(-2\sqrt{2}K  \sqrt{m_i} \Bigr). \label{e:second_moment_normal_distribution_aux}
        \end{align}
        We may assume w.l.o.g~that $M_{ii} \leq c$. Else one could just redefine for all $M_{ii}$ that satisfy the estimate 
        \begin{equation*}
          M_{ii} \geq \delta + \sum_{i,j \in \mathds{Z}, i \neq j}|M_{ij}| =: c
        \end{equation*}
        the convex part $\psi_i^c$ of the single site potential $\psi_i$ according to 
        \begin{equation*}
          \tilde \psi_i^c (x_i) = \psi_i^c(x_i) + (M_{ii} -c) x_{i}^2.
        \end{equation*}
        Then, the new $\tilde M_{ii}$ would become $\tilde M_{ii}=c$.
        Therefore we can continue the estimation of~\eqref{e:second_moment_normal_distribution_aux} as
        \begin{align}
          \Phi\Bigl(\frac{\sqrt{2}(s_i-Km_i)}{\sqrt{m_i}}\Bigr) \ge \Phi\Bigl(-2\sqrt{2}K  \sqrt{c} \Bigr) > 0. \label{e:second_moment_normal_distribution_second}
        \end{align}
	A combination of the estimate~\eqref{e:second_moment_normal_distribution_first} and~\eqref{e:second_moment_normal_distribution_second} yields
	\[
	  \frac{\Phi\bigl(\frac{\sqrt{2}(s_i-K(m_i-a))}{\sqrt{m_i-a}}\bigr)}{\Phi\bigl(\frac{\sqrt{2}(s_i-Km_i)}{\sqrt{m_i}}\bigr)} \lesssim 1.
	\]
	A similar calculation works for the other expression in~\eqref{e:second_moment_single_site_core} that involves the cumulative normal distribution function~$\Phi$.
	Hence, we have the shown the desired inequality~\eqref{e_moment_estimate_single-site}.\medskip

        \noindent Since $\sum_{j \in \mathds{Z}, j \neq i} |M_{ij}| \le m_i-a$ by Assumption~(\ref{e_strictly_diag_dominant}), convexity yields
        \begin{align*}
        	\frac{s_i^2}{m_i(m_i-a)} &
			\le \frac{\left( \sum_{\substack{j \in \mathds{Z} \\ j \neq i}} |M_{ij}| \, |x_j| \right)^2}{(m_i-a)^2} \\
			 & = \left(\sum_{j \neq i} \frac{|M_{ij}|}{m_i-a} |x_j|\right)^2
			\le \sum_{j \neq i} \frac{|M_{ij}|}{m_i-a} x_j^2.
        \end{align*}
	Combining the latter with~\eqref{e_moment_estimate_single-site} and using H\"older's inequality we obtain
	\begin{align*}
		\E_{\mu_{\Lambda}}[\e^{ax_i^2}] & \lesssim  \E_{\mu_{\Lambda}} [\e^{\frac{as_i^2}{m_i(m_i-a)}}]  \\
 			& \lesssim \E_{\mu_{\Lambda}}\Bigl[\prod_{j \neq i} \e^{\frac{|M_{ij}|}{m_i-a} ax_j^2}\Bigr]
			\le \prod_{j \neq i} \Bigl( \E_{\mu_{\Lambda}}[ \e^{ax_j^2} ] \Bigr)^{\frac{|M_{ij}|}{m_i-a}}.
	\end{align*}
	Define $v_i \coloneqq \log\E_{\mu_{\Lambda}}[\e^{a x_i^2}]$. Then the previous estimate gives
	\[
		(m_i-a) v_i \le (m_i-a) \log c + \sum_{j \neq i} |M_{ij}| v_j
	\]
	for a constant $c \ge 1$.
	By defining the matrix $A$ via 
        \begin{equation*}
        A_{ij} \coloneqq
        \begin{cases}
           m_i - a , & \mbox{for } i=j; \\
           -|M_{ij}|, &\mbox{for }	i \neq j.  
        \end{cases}
        \end{equation*}
         Then the previous estimate shows that for $v= (v_i)$ and $$\hat{c} \coloneqq \sup_i (m_i - a) \log(c) \lesssim 1$$
         the following estimate holds
	\begin{equation}\label{eq:matineq}
		Av \le \hat{c} \ \setone,
	\end{equation}
	where the last inequality has to be understood component wise. Note that matrix $A$ is strictly diagonal dominant, i.e.
        \begin{align*}
          A_{ii} - \sum_{\substack{j \in \mathds{Z} \\ j \neq i }} |A_{ij}|  \geq M_{ii} - \frac{3}{4} \delta - \sum_{\substack{j \in \mathds{Z} \\ j \neq i }} |M_{ij}|  \overset{~(\ref{e_strictly_diag_dominant})}{\geq} \frac{1}{4} \delta >0.
        \end{align*}
        Therefore it follows from a standart result (see Lemma~\ref{p:positiv_entries_of_inverses} below) that the entries $(A^{-1})_{ij}$ of the inverse $A^{-1}$ are nonnegative i.e. $(A^{-1})_{ij} \geq 0$. Hence, applying the inverse $A^{-1}$ on~\eqref{eq:matineq} yields the component wise inequality 
	\[
		v \le \hat{c} \; A^{-1} \; \setone,
	\]
        which yields the desired estimate~\eqref{e:exponential_moment} by an application of~\eqref{eq:bounded_ness_of_inverse_of_A} of Lemma~\ref{p:positiv_entries_of_inverses} from below.\medskip

	\noindent The remaining estimate~\eqref{e:arbitrary_moment} follows from the estimate~\eqref{e:exponential_moment} and the point-wise bound $$x_i^{2k} \le \frac{k!}{a^k} \e^{ax_i^2}.$$
\end{proof}

In the proof of Lemma~\ref{lem:moments} we used the following basic result on matrix theory.
\begin{lemma}\label{p:positiv_entries_of_inverses}
 Assume that the Matrix $A=(A_{ij})$ is strictly diagonally dominant in the sense of~(\ref{e_strictly_diag_dominant}) i.e. 
 \begin{align*}
   M_{ii} - \sum_{j \in \mathds{Z}, j \neq i} |M_{ij}| \geq \delta>0.
 \end{align*}
 Additionally, assume that the off-diagonal entries of $A$ are nonpositive i.e. $A_{ij} \leq 0$ for $i \neq j$.  \newline
Then the matrix $A$ is invertible, the entries of the matrix $A^{-1}$ are nonnegative i.e. $(A^{-1})_{ij} \geq 0$, and  
\begin{equation}
  \label{eq:bounded_ness_of_inverse_of_A}
    \sup_i \sum_{j=1} (A^{-1} )_{ij} \leq \frac{1}{\delta}.
\end{equation}
\end{lemma}
The fact that the entries of $A^{-1}$ are nonnegative is for example also shown in~\cite[Lemma~5]{OR07} by induction. Our proof is a bit simpler and uses analysis and semigroups.  
\begin{proof}[Proof of Lemma~\ref{p:positiv_entries_of_inverses}]

	The existence of the inverse $A^{-1}$ follows directly from the fact that $A$ is positive definite, which follows directly from the fact that $A$ is strictly diagonal dominant. \smallskip

        \noindent Let us turn to the non-negativity of the entries of the inverse $A^{-1}$. We will show that the inverse $A^{-1}$ leaves set 
	\[
		S \coloneqq \{ y \in \mathds{R}^{\Lambda} : 0 \le y_i  \text{ for all } i \in \Lambda \}
	\]
invariant, which yields the non-negativity of the entries of the inverse $A^{-1}$ as a direct consequence. 
We need two observations. The first one is that the solution $e^{-tA}$ of $\dot{y}(t) = -Ay(t)$ remains in $S$, if $y(0) \in S$. Indeed, assume $y(t)$ does not stay in $S$. Then there is a time $t^*$ such that $y(t^*) \in \partial S \subset S$. Therefore, a component $y_i(t^*)$ must be $0$ i.e.~$y_i(t^*)=0$. One directly sees that $-Ay \cdot e_i \ge 0$ since $A_{ij} \le 0$ for $i \neq j$. Hence, $y_i(t^*)$ becomes positive again. Therefore $y(t)$ is rejected from the boundary of $S$ and therefore $y(t)$ remains in $S$ for all $t \ge 0$. \newline
\noindent The second observation is that the inverse $A^{-1}$ can be represented in the following way
\begin{equation*}
  A^{-1} = \int_0^{\infty} e^{-t A} \dx t.
\end{equation*}
Indeed, it holds that
\begin{equation*}
  A  \int_0^{\infty} e^{-t A} \dx t = - \int_0^{\infty} \frac{d}{dt} e^{-tA} dt = I.
\end{equation*}
The desired statement --$A^{-1}$ leaves set $S$ invariant-- now follows directly from a combination of the two observations.\smallskip

Let us verify the estimate~\eqref{eq:bounded_ness_of_inverse_of_A}. Because the matrix $A$ is strictly diagonally dominant and the entries $A_{ij}$, $i \neq j$, are nonpositive, it holds that for any $i$
\begin{align*}
  \sum_{j} A_{ij} \geq \delta.
\end{align*}
Because the entries of $A^{-1}$ are nonnegative it follows that for an arbitrary index $k$
\begin{align*}
   1 = \sum_{i} (A^{-1})_{ki}  \sum_{j} A_{ij} \geq \delta \sum_{i} (A^{-1})_{ki} ,
\end{align*}
which already yields the desired estimate~\eqref{eq:bounded_ness_of_inverse_of_A}.
\end{proof} \medskip

Now, we can turn to the proof of Proposition~\ref{p_mr_decay_reduced_measure}. As already noted before, we deduce Proposition~\ref{p_mr_decay_reduced_measure} by combining Zegarlinki's approach~\cite{Zeg96} with a perturbation argument. For this, we pass from the Hamiltonian $H$ to a Hamiltonian with suitably truncated interaction. The reason is that Zegarlinski's approach only works if the spins between $i$ and $j$ have finite-range interaction $R$. However if the range $R$ is fixed, Zegarlinski's method yields exponential decay of covariances, which is a lot more than we would like to proof. The main idea of the argument is to let $R$ grow w.r.t.~the distance $|i-j|$. This will have two consequences: 
\begin{itemize}
\item On the one hand, the covariance estimate obtained for the measure with truncated interaction does not decay exponentially anymore but still algebraically (cf.~Lemma~\ref{lem:zeg} below).
\item On the other hand, the measure with truncated interaction becomes close to the measure $\mu_{\Lambda,q}$ (cf.~Lemma~\ref{lem:cut} below).
\end{itemize}
So in the end, we can transfer the algebraic decay of covariances to the desired measure $\mu_{\Lambda,q}$ by a combination of both observations. So, Proposition~\ref{p_mr_decay_reduced_measure} is a direct consequence of a combination of Lemma~\ref{lem:cut} and Lemma~\ref{lem:zeg} from below. \medskip

Let us now specify how the interactions  are truncated. Without loss of generality we may assume that $i < j$. We define the set $I \subset \Lambda \times \Lambda$ as 
\begin{equation}
  \label{eq:def_I}
		I \coloneqq \bigl\{ (i,j) \in \Lambda \times \Lambda : |i-j| \ge R \bigr\} \setminus \bigl( (S_1 \times S_1) \cup (S_2 \times S_2) \bigr),
\end{equation}
where the sets $S_1, S_2 \subset \Lambda$ are given by 
$$S_1\coloneqq \left\{k \in \Lambda : k <i \right\} \qquad \mbox{and} \qquad S_2\coloneqq \left\{k \in \Lambda : k >j \right\}.$$
In the approximating measure $\mu_{\Lambda,q,I}$ we drop the interactions corresponding to the index set $I$ i.e.~the measure $\mu_{\Lambda,q,I}$ is given by the density
\begin{equation}
  \label{eq:def_aproximation_mu_q}
	\dx\mu_{\Lambda,q,I}(p^\Lambda) \coloneqq \frac{1}{Z_{\mu_{\Lambda,q,I}}} \e^{-\sum_{k \in \Lambda} \psi_{k,q}(p_k) - \sum_{(k,l) \in (\Lambda \times \Lambda) \setminus I} M_{kl} p_k p_l} \; \dx p^\Lambda.
\end{equation}

Note that the sites left from $i$ and right from $j$ are still allowed to interaction with each other. The reason is that the cut-off estimate of Lemma~\ref{lem:cut} below is rather sensitive to the amount of interaction we drop. This is also the reason why we cannot restrict ourselves to finite-range interaction on the whole system. Still, the method of Zegarlinski can still be applied as long as the sites between $i$ and $j$ interact with finite-range and that there is no interaction across $i$ and $j$.\medskip

In the next statement, we show that the covariance $\cov_{\mu_{\Lambda},q}(p_i,p_j)$ is close to the covariance $\cov_{\mu_{\Lambda},q,I}(p_i,p_j)$ if the cutoff $R$ is large enough.
\begin{lemma}\label{lem:cut}
	We consider the measures $\mu_{\Lambda,q}$ and $\mu_{\Lambda,q,I}$ given by~\eqref{eq:def_mu_q} and~\eqref{eq:def_aproximation_mu_q} respectively. Recall the number $0<\alpha<\infty$ from the decay of interaction condition~\eqref{e_decay_ising}. \newline
        If we choose the cutoff $R$ in the definition~\eqref{eq:def_I} of $I$  as
        \begin{equation*}
         R=|i-j|^{1-\epsilon} 
        \end{equation*}
for some $0< \epsilon < \alpha$ small enough, then there is a number $0 < \delta < \alpha$ such that 
        \begin{equation*}
          |\cov_{\mu_{\Lambda},q}(p_i,p_j) -\cov_{\mu_{\Lambda},q,I}(p_i,p_j) | \lesssim \frac{1}{|i-j|^\delta}, 
        \end{equation*}
        uniformly in $\Lambda$ and $q^\Lambda$.
\end{lemma}
\begin{proof}[Proof of Lemma~\ref{lem:cut}]
  For $\lambda \in \left[0,1 \right]$ we introduce the auxiliary measures $\nu_\lambda$ by the density
  \begin{equation*}
  	\nu_{\lambda}(d p^\Lambda) \coloneqq \frac{1}{Z_{\nu_{\lambda}}} \e^{-\sum_{i \in \Lambda} \psi_{i,q}(p_i) - \sum_{(i,j) \in (\Lambda \times \Lambda) \setminus I} M_{ij} p_i p_j - \lambda \sum_{(i,j) \in I} M_{ij} p_i p_j} \; \dx p^\Lambda.  
  \end{equation*}
  It follows from the definition~\eqref{eq:def_mu_q} and~\eqref{eq:def_aproximation_mu_q} of the measures $\mu_{\Lambda,q}$ and $\mu_{\Lambda,q,I}$ that $$\nu_0 = \mu_{\Lambda,q,I} \qquad \mbox{and} \qquad \nu_1 =\mu_{\Lambda,q}.$$
  Therefore, the fundamental theorem of calculus yields the identity
  \begin{align*}
    |\cov_{\mu_{\Lambda},q}(p_i,p_j) -\cov_{\mu_{\Lambda},q,I}(p_i,p_j) | & = |\int p_i p_j \; \dx \mu_{\Lambda,q}(p^\Lambda) - \int p_i p_j \; \dx \mu_{\Lambda,q,I}(p^\Lambda) | \\  
    & \leq \int_0^1 |\frac{d}{d\lambda} \int p_i p_j \; \dx \nu_{\Lambda}(p^\Lambda)| \, \dx \lambda.
  \end{align*}
  Direct calculation yields that
  \begin{align*}
    \frac{d}{d\lambda} \int p_i p_j \; \dx \nu_{\Lambda}(p^\Lambda) & = \cov_{\nu_\lambda}(p_ip_j, \sum_{(k,l) \in I} M_{kl} p_k p_l) \\
    & = \sum_{(k,l) \in I} M_{kl} \cov_{\nu_\lambda}(p_ip_j, p_k p_l) .
  \end{align*}
By Hoelder's inequality and an application of the moment estimates of Lemma~\ref{lem:moments}, we see that 
\begin{align*}
  |\cov_{\nu_\lambda}(p_ip_j, p_k p_l)| & \leq \int |p_i p_j p_k p_l| \; \dx \nu_\lambda +  \int |p_i p_j| \; \dx \nu_\lambda   \int |p_k p_l| \; \dx \nu_\lambda  \lesssim 1
\end{align*}
uniformly in $\Lambda$, $q$, and $\lambda$. Therefore, the desired statement follows from the observation
\begin{equation}\label{e:approx_central_est}
  \sum_{(k,l) \in I} M_{kl} \lesssim \frac{1}{|i-j|^\delta}.
\end{equation}
Indeed, from the definition~\eqref{eq:def_I} of $I$, we have
\begin{align}
  \label{e:approx_central_est_partition}  \sum_{(k,l) \in I} M_{kl} & = \sum_{k: k < i} \sum_{l:(k,l) \in I} M_{k,l}  \\
& \qquad  + \sum_{k:i \leq k \leq j}  \sum_{l:(k,l) \in I} M_{k,l} +\sum_{k:k > j} \sum_{l:(k,l) \in I} M_{k,l} . 
\end{align}
Let us start with the estimation of the second term on the right hand side of the last identity. Using the decay property~\eqref{e_decay_ising} of the interaction $M_{kl}$ we see that
\begin{align*}
  \sum_{k: i \leq k \leq j}  \sum_{l:(k,l) \in I} M_{k,l} &\lesssim \sum_{k:i \leq k \leq j} \sum_{l:|k-l|\geq R} \frac{1}{|k-l|^{2+\alpha}} \\
  & \lesssim \sum_{k:i \leq k \leq j} \frac{1}{R^{1+\alpha}} = \frac{|i-j|}{(|i-j|^{1-\varepsilon})^{1+\alpha}} \lesssim  \frac{1}{|i-j|^{\delta}}.
\end{align*}
Now, we turn to the estimation of the first term on the right hand side of~\eqref{e:approx_central_est_partition}. Again, by using~\eqref{e_decay_ising} we have
\begin{align*}
  \sum_{k:k <i}  \sum_{l:(k,l) \in I} M_{k,l} & \lesssim \sum_{ k:k<i }  \frac{1}{|k - (i+ R)|^{1+\alpha}} \lesssim \frac{1}{R^\alpha}=  \frac{1}{|i-j|^{\alpha (1- \varepsilon)}}.
\end{align*}

 The third term can be estimated in the same way. Therefore, we have deduced the desired estimate~\eqref{e:approx_central_est} and completed the proof.
\end{proof}

In the next statement, we show by using the approach of~\cite[Lemma~4.5]{Zeg96} that the covariances of the measure $\mu_{\Lambda, q, I}$ decay algebraically. For this purpose, we have to work out the dependence on the range $R$ which was not necessary for deducing~\cite[Lemma~4.5]{Zeg96}.

\begin{lemma}\label{lem:zeg}
  Under the same assumptions as in Lemma~\ref{lem:cut}, it holds that there is a constant $\delta>0$ such that
  \begin{equation*}
    |\cov_{\mu_{\Lambda,q I}} (p_i, p_j)| = |\int p_i p_j \ \mu_{\Lambda,q I} (dp^{\Lambda}) | \lesssim \frac{1}{|i-j|^\delta} ,
  \end{equation*}
  uniformly in $\Lambda$ and $q^\Lambda$. 
\end{lemma}

\begin{proof}[Proof of Lemma~\ref{lem:zeg}]
	It follows from definition~\eqref{eq:def_aproximation_mu_q} of the measure~$\mu_{\Lambda,q, I}$ that 
        \begin{equation*}
          	\mu_{\Lambda,q,I}(d p^\Lambda) \coloneqq \frac{1}{Z_{\mu_{\Lambda,q,I}}} \e^{-\sum_{k \in \Lambda} \psi_{k,q}(p_k) - \sum_{k,l \in \Lambda} \tilde M_{kl} p_k p_l} \; \dx p^\Lambda,
        \end{equation*}
where the matrix $\tilde M =(M_{kl})$ is given by the entries
\begin{equation*}
  \tilde M_{kl} \coloneqq
  \begin{cases}
    M_{i,j}, & \mbox{if } (i,j)\in \Lambda \times \Lambda \setminus I \\
    0, & \mbox{if } (i,j)\in  I .
  \end{cases}
\end{equation*}
We decompose the sites in $\Lambda$ into $L\coloneqq \bigl[ |i-j|^\varepsilon  \bigr]$ blocks, each containing at least $R=|i-j|^{1-\varepsilon}$ many sites. For this purpose we pick a finite sequence $(b_k)_{k=1}^{L-1}$ of integers (not necessarily contained in $\Lambda$) that satisfy $i < b_1 < b_2 < \cdots < b_L \le j$ and $b_k - b_{k-1} \ge R$ for $1 \le k \le L$. This introduces a partition of $\Lambda$ into (possibly empty) blocks
	\[
		B_n \coloneqq \bigl\{ l \in \Lambda : b_n \le l < b_{n+1} \bigr\}
	\]
	for $0 \le n \le L$, where for notational simplicity set $b_0 \coloneqq -\infty$ and $b_{L+1} \coloneqq \infty$. \medskip

The key ingredient for the approach of Zagarlinski~\cite{Zeg96} is the following representation of the covariance, namely for all $0 \leq m \leq L$ 
\begin{equation}
  \label{eq:Zeg_key_repres}
		\int p_i p_j \; \dx \mu_{\Lambda,q,I} = \int p_i p_j f_m(p^\Lambda) \; \dx\mu_{\Lambda,q,I},  
\end{equation}
where the function $f_m$ is given by 
	\[
		f_m(p) \coloneqq 		= \prod_{k=1}^m \tanh\bigl(-2 \sum_{ k \in B_{k-1}} \sum_{ l \in B_k} M_{kl} p_k p_l \bigr).
	\]
We show the identity~\eqref{eq:Zeg_key_repres} by induction. For $m=0$, the identity~\eqref{eq:Zeg_key_repres} holds since $f \equiv 1$. Let us consider now $m\geq 1$. Substituting $p_i$ by $-p_i$ for $i < b_m$ and exploiting the fact that $f_{m-1}$ is invariant under this substitution yields the identity
	\begin{equation*}
		\int p_i p_j f_{m-1}(p^\Lambda) \; \dx\mu_{\Lambda,q,I} = - \int p_i p_j f_{m-1}(p^\Lambda) \e^{4\sum_{i < b_m \le j} M_{ij} p_i p_j}  \dx\mu_{\Lambda,q}.
	\end{equation*}
Using the last identity we directly get that
	\begin{align}
	& \int p_i p_j f_{m-1}(p^\Lambda) \; \dx\mu_{\Lambda,q,I} \notag \\
			& \qquad = \int p_i p_j f_{m-1}(p^\Lambda) \ \frac{1}{2}\Bigl(1 - \e^{4\sum_{k < b_m \le l} M_{kl} p_k p_l} \Bigr) \dx\mu_{\Lambda,q}. \label{e:eq:Zeg_key_repres_subst}
	\end{align}
By the assumption on the support of $\tilde M$ it holds that
\[
   \sum_{i < b_m \le j} \tilde M_{ij} p_i p_j =  \sum_{i \in B_{m-1}} \sum_{j \in B_m} M_{ij} p_i p_j.
\]
Therefore, by using the definition of $f_m$ one can rewrite the right hand side of the identity~\eqref{e:eq:Zeg_key_repres_subst} as
	\begin{align*}
          & \int p_i p_j f_{m-1}(p^\Lambda) \ \frac{1}{2}\Bigl(1 - \e^{4\sum_{k < b_m \le l} M_{kl} p_k p_l} \Bigr) \dx\mu_{\Lambda,q} \\
			& \qquad = \int p_i p_j f_m(p^\Lambda) \cdot \frac{1}{2}\Bigl(1 + \e^{4\sum_{i < b_m \le j} M_{ij}p_ip_j} \Bigr) \dx\mu_{\Lambda,q}(p) \\
        \end{align*}

 Applying now again the substitution $p_i$ by $-p_i$ for $i < b_m$ to the second summand in the last identity, one sees that 
	\begin{align*}
          & \int p_i p_j f_m(p^\Lambda) \cdot \frac{1}{2}\Bigl(1 + \e^{4\sum_{i < b_m \le j} M_{ij}p_ip_j} \Bigr) \dx\mu_{\Lambda,q}(p) \\
			& \qquad = \int p_i p_j f_m(p^\Lambda) \dx\mu_{\Lambda,q}.
	\end{align*}
Here, we also used the fact that $f_m$ becomes $- f_m$ in the last substitution. So we have overall deduced the identity  
	\begin{align*}
	& \int p_i p_j f_{m-1}(p^\Lambda) \; \dx\mu_{\Lambda,q,I} = \int p_i p_j f_m(p^\Lambda) \dx\mu_{\Lambda,q},
	\end{align*} 
which yields the desired formula~\eqref{eq:Zeg_key_repres}. \medskip



For some $T >0$ that is fixed later we consider the event
	\[
		A \coloneqq \Bigl\{ p^\Lambda  : \#\Bigl\{1 \leq n \leq L \ : \  \sum_{k \in B(n-1)}  \sum_{ l \in B(n)}|M_{kl} p_k p_l| \ge T \Bigr\} \le \frac{L}{2} \Bigr\}.
\]
In view of the identity~\eqref{eq:Zeg_key_repres}, the statement of Lemma~\ref{lem:zeg} follows if we show the following two estimates, namely
\begin{equation}
  \label{eq:lem_zeg_first_estimate}
  | \int_A p_i p_j f_L(p^\Lambda) \; \dx \mu_{\Lambda, q, I} | \lesssim \frac{1}{|i-j|^\delta} 
\end{equation}
and
\begin{equation}
  \label{eq:lem_zeg_second_estimate}
  | \int_{A^c} p_i p_j f_L(p^\Lambda) \; \dx \mu_{\Lambda, q, I} | \lesssim \frac{1}{|i-j|^\delta} .
\end{equation}
Let us first derive the estimate~\eqref{eq:lem_zeg_first_estimate}. We set $T= \frac{1}{4} \ln L^{\frac{1}{2}}$. It follows from the definition of the event $A$ that for $p^\Lambda \in A$ it holds 
	\begin{align*}
		|f(p^\Lambda)| &\le \bigl( \tanh(2T) \bigr)^{\frac{L}{2}}= e^{\frac{L}{2} \ln \left( 1- \frac{2}{\exp(4T) + 1} \right)} \\ 
                & \leq e^{-\frac{L}{\exp(4T)+1}} \lesssim e^{-L^{\frac{1}{2}}} \lesssim  e^{-|i-j|^{\frac{\varepsilon}{3}}}. 
	\end{align*}
From this point-wise estimate, the desired inequality~\eqref{eq:lem_zeg_first_estimate} follows immediately. \newline
Now, let us turn to the
estimate~\eqref{eq:lem_zeg_second_estimate}. We will show that the complementary event $A^c$ has very small probability. However, we deduce a preliminary estimate first. More precisely, we derive that there is a constant $a>0$ such that
\begin{align}
  & \int e^{\frac{a}{L} \sum_{n=1}^L \sum_{k \in B(n-1) } \sum_{l \in B(n)} |M_{kl} p_k p_l |} \; \dx \mu_{\Lambda,q, I} \lesssim 1.   \label{eq:Zegarlinski_moment}
\end{align}
Indeed, by the decay property~\eqref{e_decay_ising} of the interaction $M_{kl}$, we can estimate for any $1 \leq n \leq L$ using the ad-hoc notation $d_{k,n} := \dist(k, B(n)) $
\begin{align*}
 &  \sum_{k \in B(n-1)} \sum_{l \in B(n)} |M_{kl} p_k p_l|  \leq \sum_{k \in B(n-1)} p_k^2 \sum_{l \in B(n)} |M_{kl}| + \sum_{l \in B(n)} p_l^2 \sum_{k \in B(n-1)} |M_{kl} | \\
  & \qquad \leq C\sum_{k \in B(n-1)} \frac{1}{(d_{k,n})^{1+ \alpha}} p_k^2  + C \sum_{l \in B(n)} \frac{1}{ (d_{l,n-1})^{1+ \alpha}} p_l^2 .
\end{align*}
Hence, we can estimate
\begin{align*}
  & \int e^{\frac{a}{L} \sum_{n=1}^L \sum_{k \in B(n-1) } \sum_{l \in B(n)} |M_{kl} p_k p_l |} \; \dx \mu_{\Lambda,q, I} \\
  & \qquad \leq    \int e^{\frac{Ca}{L} \sum_{n=1}^L \sum_{k \in B(n)}  \left( \frac{1}{(d_{k,n-1})^{1+ \alpha}} + \frac{1}{(d_{k,n})^{1+ \alpha}}  \right) p_k^2} \; \dx \mu_{\Lambda,q, I} .
\end{align*}
 Because for small enough $a \ll 1$ it holds
 \begin{equation*}
   \frac{2Ca}{L} \sum_{n=1}^L \sum_{k \in B(n)}  \left( \frac{1}{(d_{k,n-1})^{1+ \alpha}} + \frac{1}{(d_{k,n})^{1+ \alpha}}  \right) \leq 1,
 \end{equation*}
an application of Hoelder's inequality and of the moment estimate of Lemma~\ref{lem:moments} yields the desired estimate~\eqref{eq:Zegarlinski_moment} i.e.
\begin{align}
  & \int e^{\frac{a}{L} \sum_{n=1}^L \sum_{k \in B(n-1) } \sum_{l \in B(n)} |M_{kl} p_k p_l |} \; \dx \mu_{\Lambda,q, I} \notag \\
  & \qquad \leq  \prod_{n=1}^L \prod_{ k \in B(n) } \left( \int e^{\frac{Ca}{2}  p_k^2} \; \dx \mu_{\Lambda,q, I} \right)^{\frac{a}{2L} \left( \frac{1}{(d_{k,n-1})^{1+ \alpha}} + \frac{1}{(d_{k,n})^{1+ \alpha}}  \right) } \lesssim 1. \notag
\end{align}

Now, we derive~\eqref{eq:lem_zeg_second_estimate}, which is the last missing ingredient for the proof. We obtain from Markov's inequality that
\begin{align*}
  \e^{\frac{1}{2}aT} \mu_{\Lambda,q}(A^c)	\le \int e^{\frac{a}{L} \sum_{n=1}^L \sum_{k \in B(n-1) } \sum_{l \in B(n)} |M_{kl} p_k p_l |} \; \dx \mu_{\Lambda,q, I} \lesssim 1.
\end{align*}
Because $$T = \frac{1}{4} \ln L^{\frac{1}{2}} \qquad \mbox{and} \qquad  L \leq \frac{1}{|i-j|^{\frac{\varepsilon}{2}} } $$ the last inequality yields
\begin{align*}
  \mu_{\Lambda,q}(A^c) \lesssim \e^{-\frac{1}{8}a \ln L^{\frac{1}{2}} } \lesssim  \left(\frac{1}{|i-j|^{\frac{\varepsilon}{2}} } \right)^{\frac{1}{8}a} =
 \frac{1}{|i-j|^{\frac{a\varepsilon}{16}} }.
\end{align*}
The last inequality already yields the desired estimate~\eqref{eq:lem_zeg_second_estimate} by a combination of the Hoelder's inequality, the observation $|f_L(p^\Lambda)|\leq 1 $, and an application of the moment estimate of Lemma~\ref{lem:moments}.
\end{proof}

\section{Improving the preliminary covariance estimate of Proposition~\ref{p_mr_decay_reduced_measure}: final step of the proof of Theorem~\ref{p_mr_decay}}\label{s_cov_est_postprocess}
This section we will complete the proof of~Theorem~\ref{p_mr_decay}. Note that in Proposition~\ref{p_mr_decay_reduced_measure} we have deduced the following decay of correlations i.e. 
  \begin{equation}
  |\cov_{\mu_{\Lambda,q,|M|} } (x_i,x_j)| \lesssim \ \frac{1}{|i-j|^{\tilde \alpha}+1}. \label{e_prelim_decay_corr}
\end{equation}
for some $0 < \tilde \alpha < \alpha$. In this section, we will use this estimate and a recursive scheme to improve the decay of correlations. More precisely, we will deduce the following statement. 
\begin{proposition}\label{p_right_decay_attract_inter}
  The measure $\mu_\Lambda,q,|M|$ has the following decay of correlations
  \begin{equation}\label{e_right_decay_associated_measure}
  \cov_{\mu_{\Lambda,q,|M|} } (x_i,x_j) \lesssim \ \frac{1}{|i-j|^{2+\hat \alpha}+1}.
\end{equation}
for some constant $0< \hat \alpha$.
\end{proposition}
Once the decay~\eqref{e_right_decay_associated_measure} is deduced, a combination of~\eqref{e:repre_covariance},~\eqref{e_repr_covariance_moment},~\eqref{eq:covariance_domination},  and~\eqref{e_right_decay_associated_measure} yields the statement of Theorem~\ref{p_mr_decay}.\medskip

As noted before, the proceeding to deduce~\eqref{e_right_decay_associated_measure} is motivated by the discrete case. In order to establish a recursive relation for the covariance, we need that the measure $\mu_\Lambda,q,|M|$ is a Lebowitz measure i.e.
\begin{lemma}\label{p_Lebowitz_inequality}
  Assume that the Hamiltonian $H: \mathds{R}^\Lambda \to \mathds{R}$ is given by 
  \begin{equation*}
    H(p^\Lambda) = \sum_{i \in \Lambda} V_i(x_i) - \sum_{i,j \in \Lambda} M_{ij} p_i p_j,
  \end{equation*}
  where the single-site potentials are symmetric and the interaction is ferromagnetic i.e.
  \begin{equation*}
    V_i(p_i)= V_i (-p_i) \qquad \mbox{and} \qquad M_{ij} \geq 0.
  \end{equation*}
  Then the associated measure $\mu(dp^\Lambda) = \frac{1}{Z_\mu} \exp(- H(p^\Lambda)) \dx p^\Lambda$ is a Lebowitz measure i.e.~it satisfies the Lebowitz inequality
  \begin{align}
       & \int p_i p_j p_k p_l \dx \mu(p^\Lambda) \leq \cov_\mu(x_i, x_j) \cov_\mu( x_k, x_l) \notag \\
       &  \qquad + \cov_\mu(x_i, x_k) \cov_\mu( x_j, x_l) + \cov_\mu(x_i, x_l) \cov_\mu( x_k, x_j).     \label{eq:Lebowitz}
  \end{align}
As a direct consequence of Lemma~\ref{p_Lebowitz_inequality} we get that the measure $\mu_{\Lambda,q,|M|}$ is a Lebowitz measure and satisfies~\eqref{eq:Lebowitz}.
 \end{lemma}
The statement of Lemma~\ref{p_Lebowitz_inequality} is rather standart and a proof can be found for example by Sylvester in~\cite{Sylvester}. \smallskip

The key tool for establishing a recursive relation for the covariance is the following simple consequence of the Lebowitz inequality, which is due to Simon (cf.~\cite[Theorem~3.1]{Simon}).
\begin{proposition}
  Let $A \subset \Lambda$. For convenience we write $n \notin A$ to indicate that $n \in \Lambda \backslash A$. Because~$\mu_{\Lambda,q,|M|}$ is a Lebowitz measure it holds for any $i,j \in \Lambda$ satisfying $i \in A$ and $j \notin A$ that
  \begin{align}
    \cov_{\mu_{\Lambda,q,|M|}} (x_i,x_j) & \leq \sum_{\substack{l \in A \\ n \notin A } } |M_{ln}| \Big[ \cov_{\mu_{\Lambda,q,|M|}} (x_i,x_l)  \cov_{\mu_{\Lambda,q,|M|}} (x_n,x_j) \notag \\
    & \qquad \qquad \qquad  + \cov_{\mu_{\Lambda,q,|M|}
} (x_i,x_n) \cov_{\mu_{\Lambda,q,|M|}} (x_l,x_j)\Big]      \label{e_simon_inequality}
  \end{align}
\end{proposition}

Now, all the preparations are done and we can turn to the proof of Proposition~\ref{p_right_decay_attract_inter}. \medskip

For the any site $k \in \Lambda$, we define the set $A_k$ by
\begin{align*}
  A_k = \left\{ l \in \Lambda \ | \ |i-l| \leq a \right\},
\end{align*}
for some integer $a>0$ which is chosen large enough later. Let us now fix a pair $i,j \in \Lambda$ that are sufficiently far away i.e. $|i-j| > a$.  Then it holds $j \notin A_i$ and an application of Simon's inequality~\eqref{e_simon_inequality} to the set $A_i$ yields
\begin{align*}
  \cov (x_i,x_j) & \leq \sum_{\substack{k \in A_i \\ n \notin A_i } } |M_{kn}| \Big[ \cov (x_i,x_k)  \cov (x_n,x_j) + \cov (x_i,x_n)
  \cov (x_k,x_j)\Big].
\end{align*}
Here and in the remaining part of this section, we wrote $\cov$ instead of $\cov_{\mu_\Lambda,q,|M| }$ for convenience. We split up the right hand side of the last identity and get 
\begin{align}\notag
  \cov (x_i,x_j) & \leq  \sum_{\substack{k \in A_i \\ n \notin A_i } } |M_{kn}|  \cov (x_i,x_n)  \cov (x_k,x_j) \\
& \qquad + \sum_{\substack{k \in A_i \\ n \notin A_i \\ n \notin A_j } } |M_{kn}| \cov (x_i,x_k) \cov (x_n,x_j) \notag \\
& \qquad + \underbrace{\sum_{\substack{k \in A_i \\ n \notin A_i \\ n \in A_j } } |M_{kn}| \cov (x_i,x_k) \cov (x_n,x_j)}_{=:R_{ij}}. \label{e_def_J_i_j}
\end{align}
It turns out that the term $R_{ij}$ decays nicely in $|i-j|$. Indeed, observe that 
\begin{align*}
  k \in A_i \quad \mbox{and} \quad  n \in A_j && \Rightarrow && |k-n| \geq |i-j|-2a.
\end{align*}
Therefore one gets for $|i-j|$ large enough the estimate
\begin{align*}
   R_{ij} & \overset{\eqref{e_decay_ising}}{\leq} \sum_{\substack{k \in A_i \\ n \notin A_i \\ n \in A_j } } \frac{C}{|k-n|^{2+\alpha} +1}  \cov (x_i,x_k) \cov (x_n,x_j) \\
   & \leq \frac{C}{|i-j|^{2+\alpha} +1} \sum_{\substack{k \in A_i \\ n \notin A_i \\ n \in A_j } }  \cov (x_i,x_k) \cov (x_n,x_j).
\end{align*}
Observe that
\begin{align}
  |\cov(x_n,x_j)| & =\left| \int x_n x_j \mu_{\Lambda,q,|M|}\right|  \notag \\
 &\leq \left(  \int |x_n|^2 \mu_{\Lambda,q,|M|}\right)^{\frac{1}{2}} \left(  \int |x_j|^2 \mu_{\Lambda,q,|M|}\right)^{\frac{1}{2}} \overset{\eqref{e:arbitrary_moment}}{\leq} C, \label{e_uniform_two_point_correlation_est}
\end{align}
 uniformly in $n$ and $j$. So using~\eqref{e_uniform_two_point_correlation_est} yields that
\begin{align}
   R_{ij} \leq \frac{C}{|i-j|^{2+\alpha} +1} \sum_{\substack{k \in A_i \\ n \notin A_i \\ n \in A_j } } 1  \leq 2a \ \frac{C}{|i-j|^{2+\alpha} +1} . \label{e_est_remainder_R_i_j}
\end{align}
Therefore, it is only left to estimate the first and the second sum on the right hand side of~\eqref{e_def_J_i_j}. This estimate turns out to be a little bit subtle and is done via a recursive scheme. Before setting up the recursion let us rewrite the first and the second sum of~\eqref{e_def_J_i_j}. Using the fact that $k \in A_i \Rightarrow k \notin A_j$, relabeling $k$ and $n$ in the second sum and considering $M_{kn}=M_{nk}$ yields
\begin{align} \notag
\cov (x_i,x_j) & \leq \sum_{\substack{k \in A_i \\ n \notin A_i \\ k \notin A_j} } |M_{kn}|  \cov (x_i,x_n)  \cov (x_k,x_j) \\
& \qquad + \sum_{\substack{n \in A_i \\ k \notin A_i \\ k \notin A_j } } |M_{kn}| \cov (x_i,x_n) \cov (x_k,x_j) \notag \\
& \qquad + R_{ij}. \label{e_d_remainder_cov}
\end{align}
For $k \neq i$ we define the coefficients
\begin{align} \label{e_def_tilde_J}
  \ J_{ik} =
  \begin{cases}
        \sum_{n \notin A_i} |M_{kn}| \cov(x_i, x_n), & \mbox{for } k \in A_i,\\
    \sum_{n \in A_i} |M_{kn}| \cov(x_i, x_n), & \mbox{for } k \notin A_i.
  \end{cases}
\end{align}
With this notation the estimate~\eqref{e_d_remainder_cov} can be written as
\begin{align}\label{e_recursive_rewritten}
  \cov (x_i,x_j) \leq \sum_{k \notin A_j } \ J_{ik}  \cov (x_k,x_j) + R_{ij},
\end{align}
which holds for every $j \notin A_i$. Because $k \notin A_j$ one can iteratively apply the estimate~\eqref{e_recursive_rewritten} and get the estimate
\begin{align*} 
  \cov (x_i,x_j) & \leq \sum_{k_{1} \notin A_j } \ J_{ik_{1}}  \cov (x_{k_{1}},x_j) + R_{ij} \\
  & \leq \sum_{k_{1} \notin A_j } \ J_{ik_1}  \left( \sum_{k_2 \notin A_j } \ J_{k_1 k_2}  \cov (x_{k_2},x_j) + R_{k_1 j}\right) + R_{ij} \\
 & = \sum_{ \substack{ k_{1} \notin A_j \\ k_2 \notin A_j} } \ J_{i k_1}  \ J_{k_1 k_2}  \cov (x_{k_2},x_j)  + \sum_{ k_{1} \notin A_j  } \ J_{i k_1}  R_{k_1 j} + R_{i j}
\end{align*}
After $l$-many iterations we get
\begin{align}
  \cov (x_i,x_j)  &
  \left. \begin{aligned}[t]
    = \sum_{ k_{1}, \ldots k_l \notin A_j } \ J_{i k_1} \
    J_{k_1 k_2} \cdots \ J_{k_{l-1} k_l} \cov (x_{k_l},x_j)
  \end{aligned} \right\} =: T 
 \notag  \\
  \label{e_covariance_recursion} &  \left.  \begin{aligned}
    & \qquad + \sum_{ k_{1}, \ldots k_{l-1} \notin A_j  } \ J_{i k_1} \cdots \ J_{k_{l-2} k_{l-1}}  R_{k_{l-1} j} \\
    & \qquad + \sum_{ k_{1}, \ldots k_{l-2} \notin A_j  } \ J_{i k_1} \cdots \ J_{k_{l-3} k_{l-2}}  R_{k_{l-2} j}\\
    & \qquad +\cdots \cdots + R_{i j}.
  \end{aligned} \right\} =: R 
\end{align}
Firstly, let us estimate the term $T$. For this purpose we need the auxiliary observation that
\begin{align}
  \label{e_est_coeff_tilde_J} 
  \sum_k \ J_{lk} \leq \tilde c <1
\end{align}
for some constant~$\tilde c >0$ by choosing the free parameter $a$ sufficiently large. Indeed, we can split up the sum into two parts, namely
\begin{align}\label{e_aux_est_coeff_tilde_J}
  \sum_k \ J_{lk} = \sum_{k \in A_i} \ J_{lk} + \sum_{k \notin A_i} \ J_{lk} .
\end{align}
We firstly estimate the sum $\sum_{k \in A_l} \ J_{lk}$. From the definition~\eqref{e_def_tilde_J} of~$\ J_{lk}$ we get
\begin{align*}
  \sum_{k \in A_l} \ J_{lk} = \sum_{k \in A_l}     \sum_{n \notin A_l} |M_{kn}| \cov(x_l, x_n)  .
\end{align*}
Using the definition of $A_l$, noting that $n \notin A_l$ and the preliminary decay~\eqref{e_prelim_decay_corr} of $\cov(x_l,x_n)$ we get  
\begin{align*}
  \sum_{k \in A_l} \ J_{lk} \leq  C \frac{1}{|a|^{\tilde \alpha}}\sum_{k \in A_l}     \sum_{n \notin A_l} |M_{kn}|  .
\end{align*}
Using now the decay~\eqref{e_decay_ising} of the interaction i.e.~$|M_{kn}| \lesssim \frac{1}{|k-n|^{2 + \alpha}+1}  $ yields
\begin{align}
  \sum_{k \in A_l} \ J_{lk} & \leq  \frac{1}{|a|^{\tilde \alpha}}\sum_{k \in A_l}     \sum_{n \notin A_l} \frac{C}{|k-n|^{2 + \alpha}+1}  \notag\\
  & \leq \frac{C}{|a|^{\tilde \alpha}} < \frac{1}{2}, \label{e_estimate_first_critical_term}
\end{align}
if we choose the free constant $a$ large enough.\newline
Now, let us turn to the estimation of the second sum in~\eqref{e_aux_est_coeff_tilde_J}. From the definition~\eqref{e_def_tilde_J} of~$\ J_{lk}$ we get
\begin{align}
  \sum_{k \notin A_l} \ J_{lk} & = \sum_{k \notin A_l}     \sum_{n \in A_l} |M_{kn}| \cov(x_l, x_n) \notag \\
  & = \sum_{k \notin A_l} \   \ \sum_{n \in A_l, |n-l| \leq \frac{a}{2}} |M_{kn}| \cov(x_l, x_n)  \notag \\
  & \qquad + \sum_{k \notin A_l}\  \    \sum_{n \in A_l, |n-l| > \frac{a}{2}} |M_{kn}| \cov(x_l, x_n) . \label{e_splitting_critical_term}
\end{align}
We consider the first sum on the right hand side of the last identity. Note that for $k \notin A_l$ 
\begin{align*}
  | n - l  | \leq \frac{a}{2} \quad \Rightarrow \quad  |k-n| \geq \frac{a}{2}.
\end{align*}
 Therefore we get
\begin{align}
  \sum_{k \notin A_l} \   \ \sum_{n \in A_l, |n-l| \leq \frac{a}{2}} |M_{kn}| \cov(x_l, x_n)    & \overset{\eqref{e_uniform_two_point_correlation_est}}{\leq} C \sum_{n \in A_l, |n-l| \leq \frac{a}{2}} |M_{kn}| \notag \\
  & \overset{~\eqref{e_decay_ising}}{\leq} C \sum_{n \in A_l, |n-l| \leq \frac{a}{2}} \frac{1}{|k-n|^{2+ \alpha}} \notag \\
  & \leq    C \frac{1}{a^{\frac{\alpha}{2}}}\sum_{n \in A_l, |n-l| \leq \frac{a}{2}} \frac{1}{|k-n|^{2+ \frac{\alpha}{2}}}  \notag \\
  & < \frac{1}{4}, \label{e_est_critical_term_1}
\end{align}
if we choose the free parameter $a$ sufficiently large. Let us now consider the second term on the right hand side of the equation~\eqref{e_splitting_critical_term}. Using the preliminary decay~\eqref{e_prelim_decay_corr} of correlations we get
\begin{align}
  \sum_{k \notin A_l}\  \    \sum_{n \in A_l, |n-l| > \frac{a}{2}} |M_{kn}| \cov(x_l, x_n) & \leq \frac{C}{a^{\tilde \alpha}} \sum_{n,k} |M_{kn}| < \frac{1}{4} \label{e_est_critical_term_2}
 \end{align}
Overall, a combination of~\eqref{e_est_critical_term_1} and~\eqref{e_est_critical_term_2} yields 
\begin{align*}
    \sum_{k \notin A_l} \ J_{lk}  < \frac{1}{2},
\end{align*}
which together with~\eqref{e_aux_est_coeff_tilde_J} and~\eqref{e_estimate_first_critical_term} implies the desired the estimate~\eqref{e_est_coeff_tilde_J}. \smallskip

Let us return to the estimation of the term $T$ of~\eqref{e_covariance_recursion}. Using the estimate~\eqref{e_uniform_two_point_correlation_est} and~\eqref{e_est_coeff_tilde_J} we get 
\begin{align*}
  T \leq C_1 \ \tilde c^l 
\end{align*}
for some constant $0 < \tilde c < 1$. If we choose $l$ large enough this estimate clearly yields the desired estimate
\begin{align} \label{e_est_T_cov_recur}
  T \leq \  C_1 \ \tilde c^l \ \leq C_2 \frac{1}{|i-j|^{2+\hat \alpha}}.
\end{align}
Because the estimate for the remainder term $R$ of of~\eqref{e_covariance_recursion} will behave bad in the the number of iterations $l$ (see estimates from below), we have to choose at the same time the number of iterations as small as possible. Note that
\begin{align*}
 C_1 \ \tilde c^l  \ \leq C_2 \frac{1}{|i-j|^{2+\hat \alpha}}  && \Leftrightarrow && \exp\left( l \log \tilde c \right) \leq \frac{C_2}{C_1} \ \frac{1}{|i-j|^{2+\hat \alpha}} \\
&& \Leftrightarrow &&  l \log \tilde c  \leq \log \frac{C_2}{C_1}  + \log \frac{1}{|i-j|^{2+\hat \alpha}} \\
&& \Leftrightarrow &&  l  \geq  \frac{1}{\log \tilde c} \log \frac{C_2}{C_1}  + \frac{1}{\log \tilde c}\log \frac{1}{|i-j|^{2+\hat \alpha}}.
\end{align*}
Because we can choose $C_1 \leq C_2$ we can choose the number of iterations $l$ to be the smallest integer larger than $\left| \frac{1}{\log \tilde c} \right| \left| \log \frac{1}{|i-j|^{2+\hat \alpha}}\right|$ i.e. 
\begin{align} \label{e_how_to_choose_l}
  l \sim \underbrace{\left| \frac{1}{\log \tilde c} \right|}_{=: C_3} \left| \log \frac{1}{|i-j|^{2+\hat \alpha}}\right|, 
\end{align}
and the estimate~\eqref{e_est_T_cov_recur} is still valid. \medskip

It is only left to estimate of the remainder term $R$ of~\eqref{e_covariance_recursion}. Firstly, let us estimate the term  
\begin{align*}
  \sum_{ k_{1}, \ldots k_{l-1} \notin A_j  } \ J_{i k_1} \cdots \ J_{k_{l-2} k_{l-1}}  R_{k_{l-1} j}.
\end{align*}
For any sequence of indexes $i, k_1, k_2, \ldots k_{l-1}, j$ there exists at least one pair $k_s, k_{s+1}$ such that 
\begin{align}
  |k_s - k_{s+1}| \geq \frac{|i-j|}{l}. \label{e_cond_dist_rec_cov}
\end{align}
Because we have chosen $l$ according to~\eqref{e_how_to_choose_l} this yields
\begin{align*}
  |k_s - k_{s+1}| \geq \frac{|i-j|}{ C_3 \left| \log \frac{1}{|i-j|^{2+\hat \alpha}}\right|} \geq |i-j|^{1-\varepsilon} \geq a,
\end{align*}
if we choose $|i-j|$ sufficiently large with respect to $ \varepsilon>0$. Hence it holds that $k_s \notin A_{k_{s+1}}$ and $k_{s+1} \notin A_{k_s}$. So one can estimate  
\begin{align*}
   \ J_{k_s k_{s+1}} & \overset{\eqref{e_def_tilde_J}}{=}  \sum_{n \in A_{k_s}} |M_{ k_{s+1} n}| \cov(x_{k_s}, x_n) \\
   & \overset{\eqref{e_uniform_two_point_correlation_est}}{\leq}  C  \sum_{n \in A_{k_s}} |M_{ k_{s+1} n}|
\end{align*}
Note that $n \in A_{k_s}$ yields 
\begin{align*}
  |k_{s+1} - n| \geq |k_{s+1} - k_s| - a.
\end{align*}
Therefore we get the estimate
\begin{align}
   \ J_{k_s k_{s+1}} & \leq  C a  \ \frac{1}{\left(|k_{s+1} - k_s| - a\right)^{2+\alpha}+1} \notag \\
   & \leq C a \ \frac{1}{\left(|i-j|^{1-\varepsilon} - a\right)^{2+\alpha}+1} \notag \\
   & \leq 2 C a \frac{1}{\left(|i-j|^{1-\varepsilon} \right)^{2+\alpha}+1} \notag \\
   & \leq 2 C a \frac{1}{|i-j|^{2+\hat \alpha}+1} \label{e_est_tilde_J}
\end{align}
for some $0 < \hat \alpha < \alpha$, if we choose $|i-j|$ large enough.\newline
Let us now turn to the estimation of 
\begin{align*}
  \sum_{ k_{1}, \ldots k_{l-1} \notin A_j  } \ J_{i k_1} \cdots \ J_{k_{l-2} k_{l-1}}  R_{k_{l-1} j} \leq \sum_{ k_{1}, \ldots k_{l-1} } \ J_{i k_1} \cdots \ J_{k_{l-2} k_{l-1}}  R_{k_{l-1} j} .
\end{align*}
Because at least one pair~$k_s, k_{s+1}$ satisfies the condition~\eqref{e_cond_dist_rec_cov} we have
\begin{align}
&  \sum_{ k_{1}, \ldots k_{l-1} } \ J_{i k_1} \cdots \ J_{k_{l-2} k_{l-1}}  R_{k_{l-1} j} \label{e_decomp_trick_large_jumps}   \\ & \leq \sum_{ \substack{ k_{1}, \ldots k_{l-1} \\ \mbox{\tiny $i$ and $k_1$ satisfy~\eqref{e_cond_dist_rec_cov}}} } \ J_{i k_1} \cdots \ J_{k_{l-2} k_{l-1}}  R_{k_{l-1} j} \notag \\
  & \quad + \sum_{ \substack{ k_{1}, \ldots k_{l-1} \\ \mbox{\tiny $k_1$ and $k_2$ satisfy~\eqref{e_cond_dist_rec_cov}}} } \ J_{i k_1} \cdots \ J_{k_{l-2} k_{l-1}}  R_{k_{l-1} j} \notag \\
& \quad + \ldots +  \sum_{ \substack{ k_{1}, \ldots k_{l-1} \\ \mbox{\tiny $k_{l-1}$ and $j$ satisfy~\eqref{e_cond_dist_rec_cov}}} } \ J_{i k_1} \cdots \ J_{k_{l-2} k_{l-1}}  R_{k_{l-1} j} . \notag
\end{align}
Each term on the right hand side can be estimated in the same way. For example let us estimate the first term. We have
\begin{align*}
&  \sum_{ \substack{ k_{1}, \ldots k_{l-1} \\ \mbox{\tiny $i$ and $k_1$
        satisfy~\eqref{e_cond_dist_rec_cov}}} } \ J_{i k_1} 
  \cdots \ J_{k_{l-2} k_{l-1}} R_{k_{l-1} j} \\
& \qquad \overset{\eqref{e_est_remainder_R_i_j}}{\leq}  C a^2 \sum_{ \substack{ k_{1}, \ldots k_{l-1} \\ \mbox{\tiny $i$ and $k_1$
        satisfy~\eqref{e_cond_dist_rec_cov}}} } \ J_{i k_1} 
  \cdots \ J_{k_{l-2} k_{l-1}} \\
& \qquad \overset{\eqref{e_est_tilde_J}}{\leq}  Ca^2  \frac{1}{|i-j|^{2+\hat \alpha}+1}  \sum_{  k_{1}, \ldots k_{l-1} } \ J_{k_1 k_2} 
  \cdots \ J_{k_{l-2} k_{l-1}} \\
& \qquad \overset{\eqref{e_est_coeff_tilde_J} }{\leq}  C a^2 \frac{1}{|i-j|^{2+\hat \alpha}+1}  .
\end{align*}
Every term on the right hand side of~\eqref{e_decomp_trick_large_jumps} can be estimated in this way. Only for the last term one has to use the decay~\eqref{e_est_remainder_R_i_j} of $R_{k_{l-1}j}$ instead of the decay~\eqref{e_est_tilde_J}  of $J_{k_s k_{s+1}}$. So overall one obtains the estimate 
  \begin{align*}
 & \sum_{ k_{1}, \ldots k_{l-1} \notin A_j  } \ J_{i k_1} \cdots \ J_{k_{l-2} k_{l-1}}  R_{k_{l-1} j}   \leq C a^2 l  \ \frac{1}{|i-j|^{2+\hat \alpha}+1}. 
\end{align*}
This estimate and analog estimates for the other terms of $R$ given by~\eqref{e_covariance_recursion} accumulates in the following estimation of $R$ i.e.
\begin{align*}
  R  & \leq \sum_{\tilde l =1}^l \tilde l C a^2 \ \frac{1}{|i-j|^{2+\hat \alpha}+1} \\ 
  & \leq C a^2 \ \frac{l (l+1)}{2} \frac{1}{|i-j|^{\frac{\hat \alpha}{2}}} \  \frac{1}{|i-j|^{2+\frac{\hat \alpha}{2}+1}}.
\end{align*}
Note that by our choice of $l$ (see~\eqref{e_how_to_choose_l}) we have 
\begin{align*}
  \frac{l (l+1)}{2} \frac{1}{|i-j|^{\frac{\hat \alpha}{2}}} \leq C,
\end{align*}
if we choose $|i-j|$ large enough. Therefore we obtained the desired estimate for the remainder term $R$ i.e.
\begin{align*}
  R \leq C a^2  \frac{1}{|i-j|^{2+ \frac{\hat \alpha}{2}}+1}.
\end{align*}
This completes the proof of Proposition~\ref{p_right_decay_attract_inter} and therefore also the proof of Theorem~\ref{p_mr_decay}.

\begin{acknowledgement}
The first author wants to thank Maria Westdickenberg (ne\'e Reznikoff), Felix Otto, Nobuo Yoshida, and Chris Henderson for the fruitful and inspiring discussions on this topic. The second author is grateful to Felix Otto for bringing this subject
to his attention. Both authors are indebted to the \emph{Max-Planck Institute for Mathematics in the Sciences} in Leipzig for funding during the years 2010 to 2012, where most of the content of this article originated.
\end{acknowledgement}

\bibliographystyle{amsalpha}

\bibliography{covariance_1d}

\providecommand{\bysame}{\leavevmode\hbox to3em{\hrulefill}\thinspace}
\providecommand{\MR}{\relax\ifhmode\unskip\space\fi MR }
\providecommand{\MRhref}[2]{%
  \href{http://www.ams.org/mathscinet-getitem?mr=#1}{#2}
}
\providecommand{\href}[2]{#2}
\begin{thebibliography}{CFMP05}

\bibitem[BHK82]{BHK82}
J.~B{\'e}llissard and R.~H{\o}egh-Krohn, \emph{Compactness and the maximal
  {G}ibbs state for random {G}ibbs fields on a lattice}, Comm. Math. Phys.
  \textbf{84} (1982), no.~3, 297--327.

\bibitem[CFMP05]{CaFeMePr}
M.~Cassandro, P.~A. Ferrari, I.~Merola, and E.~Presutti, \emph{Geometry of
  contours and {P}eierls estimates in {$d=1$} {I}sing models with long range
  interactions}, J. Math. Phys. \textbf{46} (2005), no.~5, 053305, 22.
  \MR{2143008}

\bibitem[COP09]{CaMaOrPi}
M.~Cassandro, E.~Orlandi, and P.~Picco, \emph{Phase transition in the 1d random
  field {I}sing model with long range interaction}, Comm. Math. Phys.
  \textbf{288} (2009), no.~2, 731--744. \MR{2500998}

\bibitem[Dob68]{Dobru_1}
R.~L. Dobru{\v{s}}in, \emph{Description of a random field by means of
  conditional probabilities and conditions for its regularity}, Teor.
  Verojatnost. i Primenen \textbf{13} (1968), 201--229. \MR{0231434}

\bibitem[Dob74]{Dobru_2}
\bysame, \emph{Conditions for the absence of phase transitions in
  one-dimensional classical systems}, Mat. Sb. (N.S.) \textbf{93(135)} (1974),
  29--49, 151. \MR{0386568}

\bibitem[Dys69]{Dyson}
F.~J. Dyson, \emph{Existence of a phase-transition in a one-dimensional {I}sing
  ferromagnet}, Comm. Math. Phys. \textbf{12} (1969), no.~2, 91--107.
  \MR{0436850}

\bibitem[FS82]{FroeSpenc}
J.~Fr{\"o}hlich and T.~Spencer, \emph{The phase transition in the
  one-dimensional {I}sing model with {$1/r^{2}$} interaction energy}, Comm.
  Math. Phys. \textbf{84} (1982), no.~1, 87--101. \MR{660541}

\bibitem[HM79]{HorMor}
T.~Horiguchi and T.~Morita, \emph{Upper and lower bounds to a correlation
  function for an {I}sing model with random interactions}, Phys. Lett. A
  \textbf{74} (1979), no.~5, 340--342. \MR{591328}

\bibitem[Imb82]{Imbrie}
J.~Z. Imbrie, \emph{Decay of correlations in the one-dimensional {I}sing model
  with {$J_{ij}=\mid i-j\mid ^{-2}$}}, Comm. Math. Phys. \textbf{85} (1982),
  no.~4, 491--515. \MR{677994}

\bibitem[KS68]{KellySher}
D.~G. Kelly and S.~Sherman, \emph{General {G}riffiths' inequalities on
  correlations in {I}sing ferromagnets}, J. Math. Phys. \textbf{9} (1968),
  no.~3.

\bibitem[Men13]{OR_rev}
G.~Menz, \emph{The approach of {O}tto-{R}eznikoff revisited}, ArXive (2013),
  preprint.

\bibitem[MO13]{MO}
G.~Menz and F.~Otto, \emph{Uniform logarithmic {S}obolev inequalities for
  conservative spin systems with super-quadratic single-site potential}, Annals
  of Probability \textbf{41} (2013), no.~3B, 21820--2224.

\bibitem[OR07]{OR07}
F.~Otto and M.~G. Reznikoff, \emph{A new criterion for the logarithmic
  {S}obolev inequality and two applications}, J. Funct. Anal. \textbf{243}
  (2007), no.~1, 121--157.

\bibitem[Roy07]{Roy07}
G.~Royer, \emph{An initiation to logarithmic {S}obolev inequalities}, SMF/AMS
  Texts and Monographs, vol.~14, American Mathematical Society, Providence, RI,
  2007, Translated from the 1999 French original by Donald Babbitt.

\bibitem[Rue68]{Ruelle_1}
D.~Ruelle, \emph{Statistical mechanics of a one-dimensional lattice gas}, Comm.
  Math. Phys. \textbf{9} (1968), 267--278. \MR{0234697}

\bibitem[Sim80]{Simon}
B.~Simon, \emph{Correlation inequalities and the decay of correlations in
  ferromagnets}, Comm. Math. Phys. \textbf{77} (1980), no.~2, 111--126.
  \MR{589426}

\bibitem[Syl76]{Sylvester}
G.~S. Sylvester, \emph{Inequalities for continuous-spin {I}sing ferromagnets},
  J. Statist. Phys. \textbf{15} (1976), no.~4, 327--341. \MR{0436856}

\bibitem[Yos01]{Yos01}
N.~Yoshida, \emph{The equivalence of the log-{S}obolev inequality and a mixing
  condition for unbounded spin systems on the lattice}, Ann. Inst. H.
  Poincar\'e Probab. Statist. \textbf{37} (2001), no.~2, 223--243.

\bibitem[Yos03]{Yos_2}
\bysame, \emph{Phase transition from the viewpoint of relaxation phenomena},
  Rev. Math. Phys. \textbf{15} (2003), no.~7, 765--788. \MR{2018287}

\bibitem[Zeg96]{Zeg96}
B.~Zegarlinski, \emph{The strong decay to equilibrium for the stochastic
  dynamics of unbounded spin systems on a lattice}, Comm. Math. Phys.
  \textbf{175} (1996), no.~2, 401--432.

\bibitem[Zit08]{Zitt}
P.-A. Zitt, \emph{Functional inequalities and uniqueness of the {G}ibbs
  measure---from log-{S}obolev to {P}oincar\'e}, ESAIM Probab. Stat.
  \textbf{12} (2008), 258--272. \MR{2374641}

\end{thebibliography}

\end{document}